\documentclass[12pt]{amsart}

\usepackage{lineno,hyperref}
\usepackage{amsmath,amsfonts,amsthm}
\usepackage{amssymb,latexsym}
\usepackage{cases}
\usepackage[numbers]{natbib} 
\usepackage{booktabs}

\usepackage{cleveref}
\renewcommand\appendix{\setcounter{secnumdepth}{-2}}

\usepackage{url,cancel}
\usepackage{float}

\usepackage{caption,enumitem}
\newtheorem{thm}{Theorem}[section]

\newtheorem{lem}[thm]{Lemma}
\newtheorem{prop}[thm]{Proposition}

\theoremstyle{remark}
\newtheorem{re}[thm]{Remark}
\newtheorem{defn}[thm]{Definition}

\newcommand{\ord}{\mbox{ord}}
\newcommand{\lcm}{\mbox{lcm}}
\newcommand{\Z}{{\mathbb Z}}

\newcommand{\R}{{\mathbb R}}

\newcommand{\N}{\mathbb{N}}

\makeatletter
\@addtoreset{equation}{section}
\makeatother
\numberwithin{equation}{section}

\begin{document}

\author{Zilong He}
\address{Department of Mathematics, University of Hong Kong, Pokfulam, Hong Kong}
\email{zilonghe@hku.hk}
\author{Ben Kane}
\address{Department of Mathematics, University of Hong Kong, Pokfulam, Hong Kong}
\email{bkane@hku.hk}
\title[Regular ternary polygonal forms]{Regular ternary  polygonal forms}
\thanks{The research of the second author is supported by grant project numbers 17316416, 17301317, and 17303618 of the Research Grants Council of Hong Kong SAR.  Part of the research was also conducted while the second author was supported by grant project number 17302515 of the Research Grants Council of Hong Kong SAR}
\date{\today}
\begin{abstract}
Inspired by Dickson's classification of regular diagonal ternary quadratic forms, we prove that there are no primitive regular ternary $ m $-gonal forms when $ m $ is sufficiently large. In order to do so, we construct sequences of primes that are inert in a certain quadratic field and show that they satisfy a certain inequality bounding the next such prime by a product of the previous primes, a question of independent interest.
\end{abstract}
\maketitle

\section{Introduction}

Representations of integers as sums of polygonal numbers have a long history going back to Fermat. Fermat famously conjectured that every integer may be written as the sum of $3$ triangular numbers, $4$ squares, $5$ pentagonal numbers and in general $m$ $m$-gonal numbers; Lagrange proved the four squares theorem, Gauss and Legendre independently showed the triangular number theorem, and Cauchy finally proved the general case. For $m\geq 3$ and $x\in\Z$, we denote by $p_{m}(x):=((m-2)x^2-(m-4)x)/2$
the \begin{it}$x$-th generalized $m$-gonal number\end{it} and for a sequence $a_1,\dots,a_n\in \N$ we define the \begin{it}$m$-gonal form\end{it} (or \begin{it}polygonal form\end{it} ) 
\begin{align*}
\triangle_{m,a}(x_{1}, \cdots , x_{n}):=\sum_{i=1}^{n}a_{i}p_{m}(x_{i}).
\end{align*}
Fermat's polygonal number conjecture may then be restated by saying that for $a=(1,\dots,1)$ of length $m$, the $ m $-gonal form $ \triangle_{m,a} $  is \begin{it}universal\end{it}, i.e., for every positive integer $\ell$, the Diophantine equation $\triangle_{m,a}(x)=\ell $
is solvable. More generally, let $ F $ be a field and $ R\subset F $ a ring. For an $n$-ary quadratic polynomial $ f(x_{1},\cdots,x_{n})\in F[x_{1},\cdots, x_{n}] $ and $\ell\in F$, we say that $\ell$ is \begin{it}represented by $f$\end{it} if the equation $ f(x)=\ell$ is solvable with $x\in R^n$, which we denote by $ \ell\mathop{\rightarrow}\limits_{R} f $, and not represented otherwise, which we denote by $ \ell\mathop{\not\rightarrow}\limits_{R} f $. It is well known that a natural number can be represented by the sum of three squares if and only if it is not of the form $  4^{t}(8\ell+7) $, with the restriction coming from the fact that one cannot even solve the congruence equation $x_1^2+x_2^2+x_3^2\equiv 7\pmod{8}$.
In investigating representations of integers by quadratic polynomials it is therefore natural to first exclude integers which cannot possibly be represented modulo a fixed integer and then separately investigate the ``sporadic'' integers for which congruence equations are always solvable but the equation over the integers is not solvable. In order to better encode this information, we let $ \mathbb{Z}_{p} $ be the $ p$-adic integers, with $\mathbb{Z}_{\infty}:=\mathbb{R}$ by convention. We say that $\ell $ is locally (resp. globally) represented by an $ n $-ary rational quadratic polynomial $ f $, if $ \ell \mathop{\rightarrow}\limits_{\Z_{p}}f $ for each prime $ p $ including $ p=\infty $ (resp. if $ \ell\mathop{\rightarrow}\limits_{\mathbb{Z}} f $).
 
A general principle, known as the Minkowski local-global principle, states that one should ``usually'' find a global solution whenever one finds a local solution. The aforementioned example states that the form given by the sums of three squares always obeys the Minkowksi local-global principle. This led to L. E. Dickson \cite{dickson_ternary_1926} asking which other quadratic forms always obey the local to global principle. He dubbed such forms \begin{it}regular\end{it}, starting an investigation and classification of such forms (his definition being equivalent to Jones's definition \cite[Corollary, p.\hskip 0.1cm 124]{jones_regularity_1931}). To more formally define regularity, we adopt the following definition of Chan and Ricci \cite{chan_representation_2015}.
\begin{defn}
	A quadratic polynomial $ f $ is said to be \begin{it}regular\end{it} if it globally represents all rational numbers that are locally represented by $ f $. We also call $ f $ irregular if $ f $ is not regular.
\end{defn}
It was shown by Jagy, Kaplansky, and Schiemann \cite{JKS} that there are at most 913 regular ternary (i.e., $n=3$) quadratic forms (some of these are still only conjectural, although the list has been shown to be correct by Lemke Oliver \cite{LemkeOliver} under the assumption of GRH), up to obvious repeats coming from multiplying a regular form by a fixed constant or by an invertible change of variables (more precisely, an isometry over $\Z$). It is hence natural to wonder how abundant regular $m$-gonal forms are. In order to exclude the obvious repeats mentioned above, we call an $m$-gonal form \begin{it}primitive\end{it} if $ \gcd(a_{1},\cdots,a_{n})=1 $ and its discriminant is defined by $ \prod_{i=1}^{n}a_{i} $. Chan and B. K. Oh \cite{chan_representations_2013} showed that there are only finitely many primitive regular ternary triangular forms ($ m=3 $), a result which was later extended by Chan and Ricci \cite{chan_representation_2015} to finiteness results for ternary quadratic polynomials. In this paper, we improve their results by obtaining a quantitative bound in terms of $m$ on the possible choices of $(a,b,c)$ for which $\triangle_{m,(a,b,c)}$ may be regular, leading to the following theorem.
\begin{thm}\label{thm:noregularmgonalforms}
There exists an absolute constant $C$ such that for $m>C$, there are no primitive regular ternary $ m $-gonal forms $ \triangle_{m,(a,b,c)}$ with $(a,b,c)\in \N^3$. 
\end{thm}
\begin{re}
Due to the bound in Theorem \ref{thm:noregularmgonalforms} and Chan and Ricci's results in \cite{chan_representation_2015}, there are only finitely many tuples $(a,b,c,m)\in\N^4$ with $\gcd(a,b,c)=1$ ($m\geq 3$) for which $\triangle_{m,(a,b,c)}$ is regular. It would be interesting to try to determine this finite set explicitly. There has been recent progress in this direction, as M. Kim and B.-K. Oh \cite{KimOh} have just completely determined all of the regular ternary triangular forms $ \triangle_{3,(a,b,c)} $, determining that there are precisely 49 of them (see \cite[Theorem 4.10 and Table 4]{KimOh} for a full list).
\end{re} 

In the classification of primitive regular ternary quadratic forms $ ax^{2}+by^{2}+cz^{2} $ (namely, $ \triangle_{4,(a,b,c)} $) \cite{dickson_ternary_1926,jones_regular_1939}, to rule out most of the irregular ones, Dickson made use of an inequality involving primes of a certain type \cite[Theorem 5]{dickson_ternary_1926}. To be more explicit, for a given positive integer $ b $, assume that $ p_{i}$'s  are all the odd prime numbers not represented by $ x^{2}+by^{2} $ in ascending order and choose $i_0$ such that 
\begin{align*}
p_{1}<p_{2}<\cdots<p_{i_{0}}<b<p_{i_{0}+1}< \cdots.
\end{align*}
He proved the inequality $ p_{i+1}<p_{1}p_{2}\cdots p_{i} $ holds for $ i\ge i_{0} $ \cite[footnote, p.\hskip 0.1cm 336]{dickson_ternary_1926}. To give a rough illustration how such an inequality applies to the regularity of such forms, suppose that $\triangle_{4,(1,b,c)}$ is regular and $p_{i_{0}+1}$ is locally represented. Then it must be the case that $c\leq p_{i_{0}+1}$ (since otherwise $x^2+by^2+cz^2=p_{i_0+1}$ cannot be solvable), and the inequality yields an inequality on $c$ depending on $b$ (as the $p_1,\dots,p_{i_{0}}$ are all primes smaller than $b$). Inspired by this, we deal with primitive ternary $ m $-gonal forms by virtue of analogous technical inequalities involving inert primes (see \eqref{eq34}), thereby showing Theorem \ref{thm:noregularmgonalforms}.

The paper is organized as follows. In Section \ref{sec:Earnest_trick}, we establish Lemma \ref{lem:newprimeanalytic} by Earnest's trick, which will be used to deduce the inequality \eqref{eq34} involving inert primes with additional restrictions analogous to Dickson's one. In Section \ref{sec:local_representation}, we  introduce the Watson's transformation and give the conditions on local representation by a (ternary) polygonal form. In Section \ref{sec:bound_abc}, we prove Theorem \ref{thm:noregularmgonalforms} by bounding the discriminant $ abc $.

\section{Earnest's trick}\label{sec:Earnest_trick}
Let $ k_{1}, k_{2},\ldots, k_{r} $ be pairwise relatively prime positive integers. Let $\chi_{i} $ be a Dirichlet character modulo $ k_{i} $ and $ \eta_{i}\in\{\pm 1\} $. Define
\begin{equation}\label{eq30}
\mathcal{S}_{\chi,\eta}:=\{n\in\Z: \chi_i(n)=\eta_i\hskip 0.15cm\forall i=1,\dots,r\}.
\end{equation}
For an integer $ M $ relatively prime to $ \Gamma:=\lcm(k_{1},k_{2},\cdots,k_{r}) $ and a nonnegative number $ x $ we furthermore set 
\[
S_x(H):=\#\{ n\in \mathcal{S}_{\chi,\eta}: n\in (x,x+H)\text{ and } \gcd(n,M)=1\}.
\]

Following Earnest's trick \cite[p.\hskip 0.1cm 855--856]{earnest_representation_1994}, we give an explicit bound on $ S_{x}(H) $. In order to state the bound, we require some notation. Let $ U=\{1,2\} $ and $ \mathbf{\alpha}=(\alpha_{1},\cdots,\alpha_{r}) $ be an element of the product set $ U^{r} $. Define  
$\chi_{\alpha}=\prod_{i=1}^{r}(\eta_{i}\chi_{i})^{\alpha_{i}}$. Then $ \chi_{\alpha} $ is clearly a Dirichlet character modulo $ \Gamma $. Characters $ \chi_{1},\ldots, \chi_{r} $ are said to be \begin{it}independent\end{it} if $ \chi_{\alpha} $ is a nonprincipal character for any $ \alpha\neq \beta_{0}$, where $ \beta_{0}=(2,\cdots,2) $. We also let $\omega(n)$ denote the number of distinct prime divisors of $n$ and $\phi$ denote the Euler totient function. 

\begin{lem}\label{lem:SxH_estimation}
	Suppose that $ \chi_{1},\cdots, \chi_{r} $ are independent. Then 
	\begin{align*}
	S_{x}(H)\hskip 0.1cm\ge \hskip 0.1cm H\dfrac{\phi(\Gamma M)}{\Gamma M2^{r}}-2^{\omega(\Gamma M)-r+1}-2^{\omega(M)}\dfrac{2^{r}-1}{2^{r}}\left(\dfrac{1}{3\log 3}\sqrt{\Gamma}\log \Gamma+\dfrac{13}{2}\sqrt{\Gamma}\right).
	\end{align*}
\end{lem}

We need explicit estimates for character sums before showing Lemma \ref{lem:SxH_estimation} and use a version of Polya--Vinogradov inequality proved by Bachman and Rachakonda \cite[Corollary, p.\hskip 0.1cm 66]{bachman_problem_2001}.

\begin{prop}[Bachman - Rachakonda]\label{prop:charactersum}
	Let $ k\in \mathbb{N}$. If $\chi$ is a nonprincipal character of modulus $ k $ and $x, y$ are real numbers with $ x<y $, then
	\begin{align*}
	\left|\sum_{x<n\le y}\chi(n)\right|\hskip 0.1cm\le \hskip 0.1cm  \dfrac{1}{3\log 3}\sqrt{k}\log k+\dfrac{13}{2}\sqrt{k},
	\end{align*}
independent of $x$ and $y$. 
\end{prop}
We modify Proposition \ref{prop:charactersum} slightly so that it is applicable to our situation.
\begin{lem}\label{lem:charactersumgcd}
	Let $ k,M $ be integers with $ \gcd(k,M)=g $. Then for any nonprincipal character $ \chi $ of modulus $ k $, we have
	\begin{align*}
	\bigg|\sum_{\substack{x<n\le y\\ \gcd(n,M)=1}}\chi(n)\bigg|\hskip 0.1cm\le \hskip 0.1cm  2^{\omega(M)-\omega(g)}\left(\dfrac{1}{3\log 3}\sqrt{k}\log k+\dfrac{13}{2}\sqrt{k}\right).
	\end{align*}
\end{lem}
\begin{proof}
	Let $ g_{1} $ be the least positive integer for which $ M/g_{1} $ is an integer relatively prime to $ g $ and $ M/g_{1}=p_{1}^{\gamma_{1}}p_{2}^{\gamma_{2}}\cdots p_{r}^{\gamma_{r}} $, where $ p_{1},\ldots, p_{r} $ are distinct primes. Then, by inclusion-exclusion, we have (letting $\mu$ denote the M\"obius $\mu$-function)
\[
\sum_{\substack{x<n\le y\\ \gcd(n,p_{1}^{\gamma_{1}}p_{2}^{\gamma_{2}}\cdots p_{r}^{\gamma_{r}} )>1}}\chi(n)=-\sum_{\substack{u\mid p_1\cdots p_r\\ u\neq 1}} \mu(u)\sum_{\substack{x<n\leq y\\ u\mid n}} \chi(n)=-\sum_{\substack{u\mid p_1\cdots p_r\\ u\neq 1}}\mu(u)\chi(u)\sum_{\substack{x/u<n<y/u}} \chi(n).
\]
Hence by Proposition \ref{prop:charactersum}, we have
	\begin{equation}\label{eq223}
	\begin{aligned}
		\bigg|\sum_{\substack{x<n\le y\\ \gcd(n,M/g_{1})>1}}\chi(n)\bigg|\hskip 0.1cm\le \hskip 0.1cm   2^{\omega(M/g_{1})-1} \left(\dfrac{1}{3\log 3}\sqrt{k}\log k+\dfrac{13}{2}\sqrt{k}\right).
	\end{aligned}
	\end{equation}	
	As $ g $ and $g_{1} $ have the same prime factors, we have $\omega(M/g_1)=\omega(M/g)=\omega(M)-\omega(g)$. Plugging this into \eqref{eq223} and noting that $ \chi(n)=0 $ if $ \gcd(n, g_{1})>1$, it follows that 
	\begin{align*}
	\bigg|\sum_{\substack{x<n\le y\\ \gcd(n,M)=1}}\chi(n)\bigg|\hskip 0.1cm=\hskip 0.1cm \bigg|\sum_{\substack{x<n\le y\\ \gcd(n,M/g_{1})=1}}\chi(n)\bigg|&\hskip 0.1cm= \hskip 0.1cm    \Bigg|\sum_{x<n\le y}\chi(n)-\sum_{\substack{x<n\le y\\ \gcd(n,M/g_{1})>1}}\chi(n)\Bigg|  \\
	&\hskip 0.1cm\le \hskip 0.1cm 2^{\omega(M)-\omega(g)}\left(\dfrac{1}{3\log 3}\sqrt{k}\log k+\dfrac{13}{2}\sqrt{k}\right), 
	\end{align*}
where in the last line we have again used Proposition \ref{prop:charactersum}.
\end{proof}

\vskip 0.5cm
\begin{proof}[Proof of Lemma \ref{lem:SxH_estimation}]
First note that if $ n\in \mathcal{S}_{\chi,\eta}$ (defined in \eqref{eq30}), then 
\begin{align*}
\chi_{\alpha}(n)&=\prod_{i=1}^{r}(\eta_{i}\chi_{i})^{\alpha_{i}}(n)=\prod_{i=1}^{r}(\chi_{i}(n))^{2\alpha_{i}}=1
\end{align*}
for any $ \alpha\in U^{r} $. On the other hand, if $n\notin\mathcal{S}_{\chi,\eta}$, then there exists some $ j $ for which either $\chi_j(n)=0$ or $\eta_j\chi_j(n)=-1$. In the former case, $ \chi_{\alpha}(n)=0 $ for any $ \alpha\in U^{r} $, while in the latter case we split the cases $\alpha_j=1$ and $\alpha_j=2$ to obtain that (assuming without loss of generality that $j=r$ for ease of notation)
\[
\sum_{\alpha\in U^{r}}\prod_{j=1}^r(\eta_{j}\chi_{j}(n))^{\alpha_{i}}=\sum_{\alpha\in U^{r-1}}\prod_{j=1}^{r-1}(\eta_{j}\chi_{j}(n))^{\alpha_{i}}\left(\eta_r\chi_r(n)+1\right)=0.
\]
 Hence we see that
\begin{equation*}
\sum\limits_{\alpha\in U^{r}}\chi_{\alpha}(n)=
\begin{cases}
2^{r} & \mbox{if }n\in \mathcal{S}_{\chi,\eta},  \\
0 & \mbox{if }n\notin\mathcal{S}_{\chi,\eta},
\end{cases}
\end{equation*}
and so
\begin{align*}
2^{r}S_{x}(H)&\hskip0.1cm= \hskip 0.1cm2^{r}\sum_{\substack{x<n\le x+H \\ \gcd(n,M)=1 \\ n\in \mathcal{S}_{\chi,\eta} }} 1\\
 &\hskip 0.1cm= \hskip 0.1cm  \sum_{\substack{x<n\le x+H \\ \gcd(n,M)=1}}\sum\limits_{\alpha\in U^{r}}\chi_{\alpha}(n)=\sum\limits_{\alpha\in U^{r}}\sum_{\substack{x<n\le x+H \\ \gcd(n,M)=1}}\chi_{\alpha}(n)\\
&\hskip 0.1cm= \hskip 0.1cm  \sum_{\substack{x<n\le x+H \\ \gcd(n,\Gamma M)=1}} 1+\sum_{\substack{\alpha\in U^{r}\\\alpha\not=\beta_{0}}}\sum_{\substack{x<n\le x+H \\ \gcd(n,M)=1}}\chi_{\alpha}(n).
\end{align*} 
We use the inclusion-exclusion principle to bound the first term from below by 
\begin{align*}
\sum_{\substack{x<n\le x+H \\ \gcd(n,\Gamma M)=1}} 1&\hskip 0.1cm= \hskip 0.1cmH \sum\limits_{u\mid \Gamma M}\dfrac{\mu(u)}{u}-\sum\limits_{u\mid \Gamma M}\mu(u)\left(\left\{\dfrac{x}{u}\right\}+\left\{\dfrac{x+H}{u}\right\}\right)  \\
&\hskip 0.1cm\ge \hskip 0.1cm H\dfrac{\phi(\Gamma M)}{ \Gamma M}-2^{\omega(\Gamma M)+1},
\end{align*}
where $\{y\}:=y-\lfloor y\rfloor$ denotes the \begin{it}fractional part\end{it} of $y\in\R$. Since the $\chi_j$ are independent, all of the characters in the second term are nonprincipal, and hence Lemma \ref{lem:charactersumgcd} may be used to obtain the lower bound 
\begin{align*}
\sum_{\substack{\alpha\in U^{r}\\\alpha\not=\beta_{0}}}\sum_{\substack{x<n\le x+H \\ \gcd(n,M)=1}}\chi_{\alpha}(n)\ge -2^{\omega(M)}(2^{r}-1)\left(\dfrac{1}{3\log 3}\sqrt{\Gamma}\log \Gamma+\dfrac{13}{2}\sqrt{\Gamma}\right).
\end{align*}
Combining these, we obtain
\begin{align*}
S_{x}(H)\hskip 0.1cm\ge \hskip 0.1cm H\dfrac{\phi(\Gamma M)}{\Gamma M2^{r}}-2^{\omega(\Gamma M)-r+1}-2^{\omega(M)}\dfrac{2^{r}-1}{2^{r}}\left(\dfrac{1}{3\log 3}\sqrt{\Gamma}\log \Gamma+\dfrac{13}{2}\sqrt{\Gamma}\right).
\end{align*}
\end{proof}
 
\begin{re}\label{re2}
	Given a discriminant $ D $ and $ n\in\N $, let $ p_{1}, p_{2},\ldots, p_{s}$ be the distinct odd prime divisors of $ d:=|D|$, $\nu_{0}(n):=\left(\frac{-4}{n}\right)$, $ \nu_{1}(n):=\left(\frac{8}{n}\right)$, and $ \chi_{i}(n):=(n/p_{i}) $, where $ (\cdot/p_{i}) $ is the Legendre symbol, $ i=1,\cdots,s $. Then the value of the Kronecker symbol $ (D/n) $ is determined by the value at $n$ of these characters (for $ D<0 $, see \cite[Chap. 1, \S 3, p.\hskip 0.1cm50]{cox_primes_2013}). 
	
	\vskip 0.5cm
		\makeatletter\def\@captype{table}\makeatother
		\renewcommand\tabcolsep{0.4cm}
		\hskip  -0.45cm	\begin{tabular}{ccccc}
			\toprule
			$ D>0 $     & & characters & $\Gamma$ & $ d $ \\
			\hline
			$   D \equiv 1\pmod{4}  $& &$ \chi_{1},\cdots,\chi_{s} $& $ p_{1}\cdots p_{s} $  &  $ p_{1}^{\alpha_{1}}\cdots p_{s}^{\alpha_{s}} $  \\
			\hline
			$D\equiv 0\pmod{4}$ &$ D=4k  $                &      &   &  \\
			&$ k\equiv 1\pmod{4} $ &  $ \chi_{1},\cdots,\chi_{s}  $& $ p_{1}\cdots p_{s}  $ & $ 4p_{1}^{\alpha_{1}}\cdots p_{s}^{\alpha_{s}}  $  \\
			&$ k\equiv 3\pmod{4} $ &  $ \nu_{0},\chi_{1},\cdots,\chi_{s}  $& $ 4p_{1}\cdots p_{s}  $ & $ 4p_{1}^{\alpha_{1}}\cdots p_{s}^{\alpha_{s}}  $  \\
			&$ k\equiv 6\pmod{8} $ &  $ \nu_{0}\nu_{1},\chi_{1},\cdots,\chi_{s}  $& $ 8p_{1}\cdots p_{s}  $ & $ 8p_{1}^{\alpha_{1}}\cdots p_{s}^{\alpha_{s}}  $  \\
			&$ k\equiv 2\pmod{8} $ &  $  \nu_{1},\chi_{1},\cdots,\chi_{s}  $& $ 8p_{1}\cdots p_{s}  $ & $ 8p_{1}^{\alpha_{1}}\cdots p_{s}^{\alpha_{s}}  $  \\
			&$ k\equiv 4\pmod{8} $ &  $  \nu_{0},\chi_{1},\cdots,\chi_{s}  $& $ 4p_{1}\cdots p_{s}  $ & $ 16p_{1}^{\alpha_{1}}\cdots p_{s}^{\alpha_{s}}  $  \\
			&$ k\equiv 0\pmod{8} $ &  $  \nu_{0},\nu_{1},\chi_{1},\cdots,\chi_{s}  $& $ 8p_{1}\cdots p_{s}  $ & $ 2^{5+t}p_{1}^{\alpha_{1}}\cdots p_{s}^{\alpha_{s}} $  \\
			\bottomrule
		\end{tabular}
		\vskip 0.5cm

	It is not difficult to verify that $ \Gamma/\phi(\Gamma)\le d/\phi(d) $, $ \omega(\Gamma)\le \omega(d)$, and $ \Gamma\le d $. Also, note that $ r\le \omega(\Gamma)+1 $, where $ r $ denotes the number of characters. By Lemma \ref{lem:SxH_estimation}, we see that $ S_{0}(H)>0 $ if
	\begin{multline}\label{eqn:Hbound}
	H>\dfrac{2dM2^{\omega(dM)}}{\phi(dM)}\left(\dfrac{1}{3\log 3}\sqrt{d}\log d+\dfrac{13}{2}\sqrt{d}+1\right)\\
	\ge \dfrac{\Gamma M}{\phi(\Gamma M)}\left(2^{\omega(\Gamma M)+1}+2^{\omega(M)}(2^{r}-1)\left(\dfrac{1}{3\log 3}\sqrt{\Gamma}\log \Gamma+\dfrac{13}{2}\sqrt{\Gamma}\right)\right).
	\end{multline}
\end{re}


Besides an explicit bound for $ S_{x}(H) $, we also need explicit upper bounds  for $n/\phi(n)$ and $\omega(n)$, which are given by Rosser and Schoenfeld \cite[Theorem 15]{rosser_approximate_1962} and Robin \cite[Th\'{e}or\`{e}me 12]{robin_estimation_1983}, respectively.

\begin{prop}[Rosser -- Schoenfeld, Robin]\label{prop:phi_omega_estimate}
	For $ n\ge 3 $, 
	\vskip 0.2cm	
	\noindent{\rm (i).}	$\dfrac{n}{\phi(n)}\le\dfrac{9}{5}\log\log n+\dfrac{2.51}{\log\log n}  $;
	\vskip 0.15cm
	\noindent{\rm (ii).} $\omega(n)\le \dfrac{\log n}{\log\log n}+1.45743\dfrac{\log n}{(\log\log n)^{2}}$.
	%
\end{prop}

\begin{lem}\label{lem:newprimeanalytic}
For a given non-square discriminant $ D $, let $ M $ be a positive integer satisfying $ M\ge 2 $ and $ \gcd(D,M)=1 $. Set $ d=|D| $. Then there exists some prime $ q\in (0,C_{0}d^{2/3}M^{1/6}) $ such that $ (D/q)=-1$ and $ \gcd(q,M)=1 $, where $ C_{0}:=20664 $ is a constant.
\end{lem}
\begin{proof}
 By assumption $ D\equiv 0,1\pmod{4} $ and $ D $ is not a perfect square, so $ d\ge3 $ and hence $ dM\ge 6 $. Consider the function $ f $ in terms of $d$ and $M$ given by 
	\begin{align*}
	f(d,M):=\dfrac{2dM2^{\omega(dM)}}{\phi(dM)}\left(\dfrac{1}{3\log 3}\sqrt{d}\log d+\dfrac{13}{2}\sqrt{d}+1\right).
	\end{align*}
	By \eqref{eqn:Hbound}, we have $S_{0}(H)>0$ when $ H\ge f(d,M) $. To find an appropriate $ H $, we estimate $ f(d,M) $ explicitly term by term by virtue of Proposition \ref{prop:phi_omega_estimate}  and prove that certain simple functions are nonnegative via a simple application of calculus. Precisely, $ 2^{\omega(dM)}\le 4 $ for $ 6\le dM<11 $,
		\begin{equation} \label{eq31} 
		\dfrac{1}{3\log 3}\sqrt{d}\log d+\dfrac{13}{2}\sqrt{d}+1\le 14
		d^{51/100},
		\end{equation}
		\vskip -0.4cm
		\begin{equation}\label{eq32}
			\dfrac{dM}{\phi(dM)}\le \dfrac{9}{5}\log\log (dM)+\dfrac{2.51}{\log\log (dM)}\le 6(dM)^{1/168},
		\end{equation}
		\vskip -0.4cm
		\begin{equation}\label{eq33} 
		2^{\omega(dM)}\le 2^{\dfrac{\log (dM)}{\log\log (dM)}+\dfrac{1.45743\log (dM)}{(\log\log (dM))^{2}}}\le 123(dM)^{211/1400} \hskip 0.2cm (dM\ge 11).
		\end{equation}
	It follows that $f(d,M)\le C_{0}d^{2/3}M^{1/6}$.
	Now, apply Lemma \ref{lem:SxH_estimation} with $ H= C_{0}d^{2/3}M^{1/6}   $ and $ \eta_{i}'s $ chosen so that $ \prod_{i=1}^{r}\eta_{i}=-1 $. Then $ S_{0}(H)\ge1 $. Hence there exists an integer $ N_{0}\in(0,H) $ such that $ (D/N_{0})=-1 $ and $ \gcd(N_{0},M)=1 $. Accordingly, there exists some prime $ q $ dividing $ N_{0} $ such that $ (D/q)=-1 $ and $ \gcd(q,M)=1 $, from which we conclude that $ q\le N_{0}\le H=C_{0}d^{2/3}M^{1/6}$.
\end{proof}

\section{Local representation over $ \mathbb{Z}_{p} $}\label{sec:local_representation}
\subsection{Notation and setup}
First, we introduce and collect some notation and definitions for the remaining sections. For a given discriminant $ D $, we let
\begin{align*}
\mathbb{P}(D)&:=\{q\,:\, \mbox{$ q $ is prime and } (D/q)=-1\},
\end{align*}
where $ (D/\cdot) $ is the Kronecker symbol. Fix an integer $m>3$ and let $\mathbb{H}$ denote a hyperbolic plane. For $ \ell,\ell_{1},\ell_{2}\in\mathbb{N} $, we define the sets
\begin{align*}
P(\ell):=\;&\mbox{all the prime factors of $ \ell $}, \\
P_{D}(\ell):=\;&\mbox{all the prime factors of $ \ell $ in $ \mathbb{P}(D) $}, \\
P_{m}(\ell_{1},\ell_{2}):=\;&P_{-4\ell_{2}}(\ell_{1})\backslash P(m-2),\\
G_{m}(\ell_{1},\ell_{2}):=\;&P(\gcd(\ell_{1},\ell_{2}))\backslash P(2(m-2)).
\end{align*}
For given positive integers $ a,b$ and $c $, write
\begin{align*}
P_{m}(a,b,c):=\;& P_{m}(a,bc)\cup P_{m}(b,ac)\cup P_{m}(c,ab),\\
G_{m}(a,b,c):=\;&G_{m}(a,b)\cup G_{m}(a,c)\cup G_{m}(b,c),
\intertext{and define the subset $ B_{m}(a,b,c) $ of $ G_{m}(a,b,c) $ by}
B_{m}(a,b,c):=\;&\{p:\,\text{$ \langle a,b,c\rangle_{p} $ is split by $ \mathbb{H} $} \}\subseteq G_{m}(a,b,c).
\end{align*}
It is not difficult to see that $ P_{m}(a,b,c)\cap G_{m}(a,b,c)=\emptyset $. Also, set
\begin{align*}
P_{m-2}:=\;&\mbox{the product of all primes in $ P(m-2)\backslash \{2\} $}, \\ 
P_{ab}:=\;&\mbox{the product of all primes in $P_{m}(a,bc)\cup P_{m}(b,ac)$}, \\ 
P_{ab}^{\prime}:=\;&\mbox{the product of all primes in $ (P_{m}(a,bc)\cup P_{m}(b,ac))\cap P(m-4)$},\\ 
P_{c}:=\;&\mbox{the product of all primes in $P_{m}(c,ab)$}, \\ 
P_{abc}:=\;&\mbox{the product of all primes in $P_{m}(a,b,c)$},
\end{align*}
and the corresponding product to be $ 1 $ if the specified set is empty. 
Put $\rho(\ell):=2^{\omega(\ell)}\ell/\phi(\ell)$ and $ K(a,b,c):=24P_{ab}\rho(P_{abc})$ for short. For convenience, we also let $ \delta=1 $ if $ \ord_{2}(m)\ge 2 $ and $ 0 $ otherwise, and introduce the notation $ \{2\}^{\delta} $ to mean the set $ \{2\} $ if $ \delta=1 $ and $ \emptyset $ otherwise.

The regularity of an $ m $-gonal form $ \triangle_{m,(a_{1},\cdots, a_{s})}$ is closely related to the quadratic form with congruence conditions given by
\[ \varphi_{m,(a_{1},\dots,a_{s})}(x_{1},\dots,x_{s}):=\sum_{i=1}^{s}a_{i}(2(m-2)x_{i}-(m-4))^{2} \]
that arises from completing the square. In this paper we are particularly interested in the case $s=3$. Now we introduce the regularity of such ternary quadratic polynomials, following the definition of B.-K. Oh \cite{oh_representations_2011}.
\begin{defn}\label{def3}
	Let $h$ be a positive integer and $ n$ and $k$ nonnegative integers. If a quadratic polynomial $f $ globally represents all nonnegative integers of the form $ hn+k$ that are locally represented by $ f $, then it is said to be $(h,k)$-regular.
\end{defn}

\begin{re}\label{re1}
	For $ a,b,c\in\mathbb{N} $, $ \triangle_{m,(a,b,c)} $ is regular if and only if $ \varphi_{m,(a,b,c)} $ is  $ (h,k) $-regular, where	$(h,k)=(8(m-2),(m-4)^{2}(a+b+c))$. Note that if $ n $ is locally represented by $ \varphi_{m,(a,b,c)} $, then $ n\equiv (m-4)^{2}(a+b+c) \pmod{8(m-2)}$. Hence $ \varphi_{m,(a,b,c)} $ is regular if and only if $ \varphi_{m,(a,b,c)} $ is $ (h,k) $-regular. Thus we also call $ \varphi_{m,(a,b,c)} $ regular instead of $ (8(m-2),(m-4)^{2}(a+b+c)) $-regular.
\end{re}

\subsection{Watson Transformations}\label{sec:wt}
Following the definitions in \cite{chan_representation_2015} and \cite{ricci_finiteness_nodate}, let $ L$ and $K $ be $ \mathbb{Z} $-lattices on nondegenerate quadratic spaces $ (V,Q) $ and $ (U,Q) $ over $ \mathbb{Q} $, respectively and $ v,u\in V $. A set $ L+v $ is called a \textit{$ \mathbb{Z} $-coset} (or a \textit{lattice translation}); it is called \textit{integral} if $ Q(L+v)\subseteq \mathbb{Z} $. Given a $ \mathbb{Z} $-coset, denote by $ n(L+v) $ the $ \mathbb{Z} $-ideal generated by $ Q(x+v) $ for all $ x\in L $, and call it \textit{primitive} if $ n(L+v)\subseteq \mathbb{Z} $. Clearly, a $ \mathbb{Z} $-coset $ L+v $ that is primitive must be integral.  Two $ \mathbb{Z} $-cosets $ L+v $ and $ K+u $ are \textit{isometric} if there exists an isometry $ \sigma:V\to U $ such that $ \sigma(L)=K $ and $ \sigma(v)-u\in K $. The \textit{conductor} of a $ \mathbb{Z} $-coset is defined by the smallest positive integer $ \mathfrak{c} $ such that $ \mathfrak{c}v\in L $. For $ n\in \mathbb{Q} $, $ n $ is said to be \textit{represented} by a $ \mathbb{Z} $-coset $ L+v $ if there exists $ x\in L $ such that $ Q(x+v)=n $. Let $ L_{p} $ be the localization of $ L $ at $ p $. The representation of $ n\in\mathbb{Q}_{p} $ by a $ \mathbb{Z}_{p} $-coset $ L_{p}+v $ is defined in the same manner. A $ \mathbb{Z} $-coset $ L+v $ is said to be \textit{regular} if it represents all rational numbers that are represented by $ L_{p}+v $ for each prime $ p $, including $ \infty $.

Similar to the case of quadratic forms and lattices, there exists a one-to-one correspondence between the set of equivalence classes of primitive regular complete quadratic polynomials in $ n $ variables over $ \mathbb{Q} $ and the set of isometry classes of primitive regular $ \mathbb{Z} $-cosets on quadratic spaces of dimension $ n $ over $ \mathbb{Q} $ (\cite[p.\hskip 0.1cm 12]{chan_representations_2013} or \cite[p.\hskip 0.1cm 84]{chan_representation_2015}). Hence we have the corresponding concepts for quadratic polynomials (e.g. conductor, integrity, primitivity and completeness, see \cite[p.\hskip 0.1cm 77]{chan_representation_2015}). We only introduce the equivalence here.

\begin{defn}
	Two quadratic polynomials $ f(x)$ and $g(x)$ over $ \mathbb{Q} $ in $ n $ variables are said to be \textit{equivalent} if there exist $ T\in GL_{n}(\mathbb{Z}) $ and $ v\in \mathbb{Z}^{n}  $ such that $ g(x)=f(xT+v)$. 
\end{defn}

Suppose that $ L $ is a ternary $ \mathbb{Z} $-lattice on a quadratic space $ (V,Q) $. As usual, we denote by $ d(L)$ the discriminant and $ n(L) $ the norm of $ L $. For any positive integer $ m $, define
\[ \Lambda_{m}(L):=\{x\in L: Q(x+z)\equiv Q(z)\pmod{m}\;\text{for all $z\in L$}\} \]
and
\[ \Lambda_{m}(L_{p}):=\{x\in L_{p}: Q(x+z)\equiv Q(z)\pmod{m}\;\text{for all $z\in L_{p}$}\}\]
for each prime $ p $.
Let $ p $ be a odd prime. If $ p\nmid n(L) $, define the maps

\begin{equation*}
\lambda_{p}(L):=
\begin{cases}
\Lambda_{p}(L)^{1/p}   & \text{if $ n(\Lambda_{p}(L))=pn(L)$,}   \\
\Lambda_{p}(L)^{1/p^{2}}   & \text{if $ n(\Lambda_{p}(L))=p^{2}n(L) $,}   \\
\end{cases}
\end{equation*}
then $ \lambda_{p} $ sends $ L $ to another lattice on the scaled space $ V^{1/p} $ or $ V^{1/p^{2}} $. Such maps $ \lambda_{p}$ are called Watson's transformations. We require several properties of $ \Lambda_{m}(L) $ and $ \Lambda_{m}(L_{p}) $ (see \cite[Lemma 4.2]{chan_representation_2015}) and a basic fact (\cite[Lemma 2.5]{ricci_finiteness_nodate}).

\begin{lem}\label{lem:LambdaL_property}
	Let $ L $ be a $ \mathbb{Z} $-lattice, $ m $ an integer and $ p $ a prime. Then
	
	{\rm (i)} $ \Lambda_{m}(L) $ is a sublattice of $ L $ and $ \Lambda_{m}(L_{p}) $ is a sublattice of $ L_{p} $.
	
	{\rm (ii)} $ \Lambda_{m}(L)_{p}=\Lambda_{m}(L_{p}) $.
	
	{\rm (iii)} $ \Lambda_{m}(L_{p})=L_{p} $ for $ p\nmid m $.
	
	{\rm (iv)} $ n(\Lambda_{m}(L))\subseteq m\mathbb{Z} $ and $ n(\Lambda_{m}(L_{p}))\subseteq p\mathbb{Z}_{p} $.
	
\end{lem}

\begin{lem}\label{lem:localgloballatticetranslation}
	Let $ L+v $ and $ K+u $ be $ \mathbb{Z}$-cosets. If $ L_{p}+v\subseteq K_{p}+u $ for all primes $ p $, then $ L+v \subseteq K+u$. In particular, $ L+v=K+u $ if and only if $ L_{p}+v=K_{p}+u$ for all primes $ p $.
\end{lem}

 The following lemma allows us to reduce the power of some prime factors of $ d(L) $ by such transformation (\cite[Lemma 4.4]{chan_representation_2015} or \cite[Lemma 2.5]{chan_discriminant_2004}).

\begin{lem}\label{lem:wt_lattice}
	Let $ L $ be a ternary $ \mathbb{Z} $-lattice and $ p $ an odd prime. If $ p^{2}\mid d(L) $, then $ d(\lambda_{p}(L))=d(L)/p^{t} $ for some $ t\in \{1,2,4\} $.
\end{lem}

\begin{re}\label{re:split}
	For a $ \mathbb{Z} $-coset $ L+v $ of conductor $ \mathfrak{c} $, if $ L_{p}$ is split by a hyperbolic plane $ \mathbb{H} $ for an odd prime $ p $ not dividing $ \mathfrak{c} $, then $ L_{p}+v=L_{p} $ represents all of the integers in $ \mathbb{Z}_{p} $.
\end{re}

By Lemma \ref{lem:LambdaL_property} (i), (ii) and (iv), $ \Lambda_{p}(L)_{p}=\Lambda_{p}(L_{p})\subseteq \{x\in L_{p}:Q(x)\in p\mathbb{Z}_{p}\} $ and when $ p $ is odd, the converse containment follows from \cite[Lemma 3.1]{chan_discriminant_2004} under the assumptions that $ p^{2}\mid d(L) $ and $ L_{p} $ is not split by $ \mathbb{H} $. Hence we have the following.

\begin{lem}\label{lem:LambdapLp}
	Let $ L $ be a ternary $ \mathbb{Z} $-lattice and $ p $ an odd prime. If $ p^{2}\mid d(L) $ and $ L_{p} $ is not split by $ \mathbb{H} $, then $ \Lambda_{p}(L)_{p}=\Lambda_{p}(L_{p})=\{x\in L_{p}:Q(x)\in p\mathbb{Z}_{p}\} $.
\end{lem}

For $ \mathbb{Z} $-cosets, we have a result analogous to Lemma \ref{lem:wt_lattice}, which is proved by Chan and Ricci \cite[Proposition 4.6]{chan_representation_2015} (or \cite[Lemma 2.6]{ricci_finiteness_nodate}). From its proof and Lemma \ref{lem:LambdapLp},  we see that the condition ``$ L_{p}+v $ does not behave well at $ p $" in \cite[Proposition 4.6]{chan_representation_2015} can be replaced by ``$p^{2}\mid d(L)$ and $ L_{p} $ is not split by $ \mathbb{H} $". Hence we are able to reformulate their proposition and prove it by following their arguments.


\begin{lem}\label{lem:wt_latticetranslation}
	Let $ L+v $ be a primitive regular ternary $ \mathbb{Z} $-coset with conductor $ \mathfrak{c} $ and $ p $ an odd prime with $ p\nmid \mathfrak{c}$. Suppose that $ p^{2}\mid d(L) $ and $ L_{p} $ is not split by $ \mathbb{H} $. Then $ \lambda_{p}(L)+p^{j}v $ is a primitive regular $ \mathbb{Z} $-coset of conductor $ \mathfrak{c} $, where $ j $ is the order of $ p $ modulo $ \mathfrak{c} $.
\end{lem}

\begin{proof}
	  Let $ L $ be on the quadratic space $ (V,Q) $ and $ j $ the order of $ p $ modulo $ \mathfrak{c} $. We assert that
	 \begin{equation}\label{eq:Lambda}
	 \Lambda_{p}(L)_{q}+p^{j}v=
	 \begin{cases}
	 L_{q}+v  & \text{if $ q\mid \mathfrak{c} $,} \\
	 \Lambda_{p}(L)_{q}      & \text{if $ q=p $,} \\
	 L_{q}  & \text{if $ q\nmid p\mathfrak{c}$.}
	 \end{cases}
	 \end{equation}
	  For $ q\mid \mathfrak{c} $, since $ p\nmid \mathfrak{c} $ and $ p^{j}v-v\in L_{q} $, $\Lambda_{p}(L)_{q}+p^{j}v=L_{q}+p^{j}v=L_{q}+v$ by Lemma \ref{lem:LambdaL_property} (iii). For $ q\nmid p\mathfrak{c} $, $\Lambda_{p}(L)_{q}+p^{j}v=\Lambda_{p}(L)_{q}=L_{q}$ by Lemma \ref{lem:LambdaL_property} (iii) again. For $ q=p $, since $ p^{2}\mid d(L) $ and $ L_{p} $ is not split by $ \mathbb{H} $,  $ \Lambda_{p}(L_{p})=\{x\in L_{p}:Q(x)\in p\mathbb{Z}_{p}\} $ by Lemma \ref{lem:LambdapLp}. Clearly, $ Q(p^{j}v)\in p\mathbb{Z}_{p} $ and so $ p^{j}v\in \Lambda_{p}(L_{p})=\Lambda_{p}(L)_{p} $. Hence $ \Lambda_{p}(L)_{p}+p^{j}v=\Lambda_{p}(L)_{p}$. Therefore, \eqref{eq:Lambda} is proved.
	   
	   Suppose that $ n$ is represented by the genus of $ \Lambda_{p}(L)+p^{j}v $. By \eqref{eq:Lambda}, $\Lambda_{p}(L)_{q}+p^{j}v=L_{q}+v$ for $ q\mid \mathfrak{c} $ or $ q\nmid p\mathfrak{c} $. By Lemma \ref{lem:LambdaL_property} (i) and (ii), $ \Lambda_{p}(L)_{p}=\Lambda_{p}(L_{p})\subseteq L_{p}=L_{p}+v $ and hence $\Lambda_{p}(L)_{q}+p^{j}v\subseteq L_{q}+v$ for each prime $ q $. So $ n$ is represented by the genus of $ L+v $. Since $ L+v $ is regular, $n $ is represented by $ L+v $. Therefore, $ n=Q(x+v)$ for some $ x\in L $. Since $ x+v\in L_{q} $ for $ q\nmid p\mathfrak{c} $ and $ x+v\in L_{q}+v $ for $ q\mid \mathfrak{c} $, it follows from \eqref{eq:Lambda} that $ x+v\in \Lambda_{p}(L)_{q}+p^{j}v $ for $ q\not=p $. For $ q=p $, since $ n $ is represented by $ \Lambda_{p}(L)_{p}+p^{j}v=\Lambda_{p}(L)_{p} $, $ p\mid n $ by Lemma \ref{lem:LambdapLp}. It follows that $ p\mid Q(x+v)$ and so $ x+v\in \Lambda_{p}(L)_{p} $ by Lemma \ref{lem:LambdapLp} again. Hence $ x+v\in  \Lambda_{p}(L)_{p}+p^{j}v $ by \eqref{eq:Lambda}. Thus $ x+v\in \Lambda_{p}(L)_{q}+p^{j}v $ for each prime $ q $ and so $ x+v\in \Lambda_{p}(L)+p^{j}v $ by Lemma \ref{lem:localgloballatticetranslation}. Therefore, $ \Lambda_{p}(L)+p^{j}v $ is regular. 
	   
	    Since scaling of $ \Lambda_{p}(L) $ preserves the conductor and the regularity, $ \lambda_{p}(L)+p^{j}v$ is of conductor $ \mathfrak{c} $ and regular. Suppose that $ n(\Lambda_{p}(L))=p^{i}n(L) $ for some $ i\in\{1,2\}$. By the definition of $ \lambda_{p} $, $ \lambda_{p}(L) $ is a $ \mathbb{Z} $-lattice on the quadratic space $ (V,Q^{\prime}) $, where $ Q^{\prime}(x)=p^{-i}Q(x) $. By \eqref{eq:Lambda},
	    \begin{equation*}
	    \Lambda_{p}(L)_{q}+p^{j}v=
	    \begin{cases}
	    \mathbb{Z}_{q}  & \text{if $ q\mid \mathfrak{c} $ or $ q\nmid p\mathfrak{c}$,} \\
	    p^{i}\mathbb{Z}_{p} & \text{if $ q=p $,} \\
	    \end{cases}
 	   \end{equation*}
	    and hence $ n(\lambda_{p}(L)+p^{j}v)=\mathbb{Z} $, showing the primitivity. 
\end{proof}

Given a primitive regular $ \mathbb{Z} $-coset $ L+v $ and an odd prime $ p\nmid \mathfrak{c} $ for which $ p^{2}\mid d(L) $ and $ L_{p} $ is not split by $ \mathbb{H} $, we are able to iteratively obtain primitive regular $ \mathbb{Z} $-cosets of conductor $ \mathfrak{c} $ until $ p^{2}\nmid d(L) $ or $ L_{p} $ is split by $ \mathbb{H} $ by applying Lemma \ref{lem:wt_latticetranslation} repeatedly, say $ \lambda_{p}^{\ell}(L)+p^{t}v $, where $ t>0$, $\ell \in \mathbb{N} $, and $ p^{t}\equiv 1\pmod{\mathfrak{c}} $. Also, $ d(\lambda_{p}^{\ell}(L))\mid d(L) $ by Lemma \ref{lem:wt_lattice}. We define the successive operations above by $ \tau_{p}(L+v):=\lambda_{p}^{\ell}(L)+p^{t}v $.

\begin{lem}\label{lem:wt_mgonalforms}
	Let $ m\ge 3 $ be a fixed integer. Given a primitive regular ternary $ m $-gonal form $\triangle_{m,(a,b,c)} $ associated with $ G_{m}(a,b,c)\not=B_{m}(a,b,c) $, there exists a primitive regular form $\triangle_{m,(a^{\prime},b^{\prime},c^{\prime})} $ such that $ a^{\prime}b^{\prime}c^{\prime}\mid abc $ and $ G_{m}(a^{\prime},b^{\prime},c^{\prime})=B_{m}(a^{\prime},b^{\prime},c^{\prime})$. 
\end{lem}
\begin{proof}
	Fix $ m\ge 3 $, clearly a ternary $ m $-gonal form $\triangle_{m,(a,b,c)} $ represents $ n $ if and only if $ \varphi_{m,(a,b,c)} $ represents $ 8(m-2)n+(m-4)^{2}(a+b+c) $. Then we associate $ \varphi_{m,(a,b,c)} $ with a $ \mathbb{Z} $-coset $ L+v $ on the quadratic space $ (Q,V) $ over $ \mathbb{Q} $; that is
	$L\cong \langle \mathfrak{c}(m-4)a/2,  \mathfrak{c}(m-4)b/2, \mathfrak{c}(m-4)c/2\rangle$ under the standard basis $ \{e_{1},e_{2},e_{3}\} $ and $ v=-d(e_{1}+e_{2}+e_{3})/\mathfrak{c}\in V$,
	where $d=(m-4)/\gcd(m-4,2(m-2)) $. One can check that $ \varphi_{m,(a,b,c)} $ represents $ 8(m-2)n+(m-4)^{2}(a+b+c) $ if and only if $ L+v $ represents
	\begin{equation*}
	hn+k:=
	\begin{cases}
	4\mathfrak{c}n+Q(v)  & \text{if $ \ord_{2}(m)=0 $,} \\
	2\mathfrak{c}n+Q(v)      & \text{if $ \ord_{2}(m)=1$,} \\
	\mathfrak{c}n+Q(v)  & \text{if $ \ord_{2}(m)>1 $,}
	\end{cases}
	\end{equation*}
	where the conductor $ \mathfrak{c} $ of $ L+v $ is given by
	\begin{equation*}
	\mathfrak{c}=
	\begin{cases}
	2(m-2)  & \text{if $ \ord_{2}(m)=0 $,} \\
	m-2     & \text{if $ \ord_{2}(m)=1$,} \\
	(m-2)/2 & \text{if $ \ord_{2}(m)>1 $.}
	\end{cases}
	\end{equation*}
	Hence we always have $ \mathfrak{c}\mid 2(m-2) $.
	
	Assume that $\triangle_{m,(a,b,c)} $ is primitive and regular. Then by the relation above, we see that $ \varphi_{m,(a,b,c)} $ is primitive and $(8(m-2),(m-4)^{2}(a+b+c))$-regular (and so it is regular by Remark \ref{re1}). Hence $ L+v $ is primitive and regular. We let 
\[
G_{m}(a,b,c)\backslash B_{m}(a,b,c)=\{p_{1},\cdots,p_{s}\}.
\]
Then for $ i=1,\ldots,s $, $ p_{i}\nmid \mathfrak{c} $ follows from $ p_{i}\nmid 2(m-2) $. Applying the operation $ \tau:=\tau_{p_{1}}\circ\cdots \circ \tau_{p_{s}} $ to $ L+v $, Lemma \ref{lem:wt_latticetranslation} implies that we obtain a primitive regular $ \mathbb{Z} $-coset $ K+u $ of conductor $ \mathfrak{c} $. Also, $ d(K)\mid d(L) $ and $ p^{2}\nmid d(K) $ or $ K_{p} $ is split by $ \mathbb{H} $ for each $p\in G_{m}(a,b,c)\backslash B_{m}(a,b,c) $. The idea is to relate the regularity of $ K+u $ to the regularity of another $ \triangle_{m,(a^{\prime},b^{\prime},c^{\prime})} $. 

In order to obtain a connection with regularity of a form $ \triangle_{m,(a^{\prime},b^{\prime},c^{\prime})} $, we first need to show that representation by $ K+u $ corresponds to representation by some $ \varphi_{m,(a^{\prime},b^{\prime},c^{\prime})} $. For this, let \[ K\cong \langle \mathfrak{c}(m-4) a^{\prime}/2,\mathfrak{c}(m-4)b^{\prime}/2,\mathfrak{c}(m-4)c^{\prime}/2\rangle\] and $ u=\ell v $, where $ \ell=p_{1}^{t_{1}}\cdots p_{s}^{t_{s}}$ and $p_{i}^{t_{i}}\equiv 1\pmod{\mathfrak{c}}$. Then $ \ell\equiv 1\pmod{\mathfrak{c}} $. Put $ \ell=1+\mathfrak{c}\ell_{0} $ and thus $ u=\mathfrak{c}\ell_{0}v+v $. Hence the quadratic polynomial $Q_{K+u} $ associated with the $ \mathbb{Z} $-coset $ K+u $ is given by
	\begin{align*}
	Q_{K+u}(x,y,z)=&\dfrac{\mathfrak{c}(m-4)}{2}\left(a^{\prime}\left(x-\ell_{0}d-\dfrac{d}{\mathfrak{c}}\right)^{2}+b^{\prime}\left(y-\ell_{0}d-\dfrac{d}{\mathfrak{c}}\right)^{2}+c^{\prime}\left(z-\ell_{0}d-\dfrac{d}{\mathfrak{c}}\right)^{2}\right).
\textsc{}	\end{align*}
	Consider the quadratic polynomial $ \varphi_{m,(a^{\prime},b^{\prime},c^{\prime})}$, which satisfies
	\begin{align*}
	\varphi_{m,(a^{\prime},b^{\prime},c^{\prime})}(x,y,z)=&\dfrac{\mathfrak{c}(m-4)}{2}\left(a^{\prime}\left(x-\dfrac{d}{\mathfrak{c}}\right)^{2}+b^{\prime}\left(y-\dfrac{d}{\mathfrak{c}}\right)^{2}+c^{\prime}\left(z-\dfrac{d}{\mathfrak{c}}\right)^{2}\right).
	\end{align*}
	Since $ \ell_{0}d\in\mathbb{Z} $, $ Q_{K+u}(x,y,z)=\varphi_{m,(a^{\prime},b^{\prime},c^{\prime})}(x-\ell_{0}d,y-\ell_{0}d,z-\ell_{0}d) $ and so they are equivalent. Since $Q_{K+u}$ is primitive and regular, so is $ \varphi_{m,(a^{\prime},b^{\prime},c^{\prime})} $. Again using Remark \ref{re1}, we see that $ \varphi_{m,(a^{\prime},b^{\prime},c^{\prime})} $ is $ (8(m-2),(m-4)^{2}(a+b+c))$-regular in particular and so $\triangle_{m,(a^{\prime},b^{\prime},c^{\prime})} $ is regular. Also, $ \triangle_{m,(a^{\prime},b^{\prime},c^{\prime})} $ is clearly primitive.
	Define the corresponding set $ G_{m}(a^{\prime},b^{\prime},c^{\prime}) $ for $ \varphi_{m,(a^{\prime},b^{\prime},c^{\prime})}  $. Next, we show
	\begin{equation}\label{eq:incluGB}
	\begin{aligned}
	G_{m}(a^{\prime},b^{\prime},c^{\prime})\backslash B_{m}(a^{\prime},b^{\prime},c^{\prime})\subseteq G_{m}(a,b,c)\backslash B_{m}(a,b,c).
	\end{aligned}
	\end{equation}
	Let $ p\in G_{m}(a^{\prime},b^{\prime},c^{\prime}) $. Then $ p\nmid 2(m-2) $ and so $ p\nmid \mathfrak{c} $. One can check that
	\begin{align*}
	p^{2}\mid a^{\prime}b^{\prime}c^{\prime}\mid d(K)\mid d(L)=\mathfrak{c}^{3}(m-2)^{3}abc/64.
	\end{align*}
	This implies that $ p^{2}\mid abc $ and so $ p\in G_{m}(a,b,c) $, and hence  $G_{m}(a^{\prime},b^{\prime},c^{\prime})\subseteq G_{m}(a,b,c)$. Note that $ B_{m}(a,b,c)\subseteq B_{m}(a^{\prime},b^{\prime},c^{\prime}) $ by \cite[Lemma 2.6]{chan_discriminant_2004} (or \cite[Lemma 2.7]{chan_discriminant_2004} with $ \beta=0 $ and $ \gamma\ge 2 $). Combining these, we have 
\[
B_{m}(a,b,c)\subseteq B_{m}(a^{\prime},b^{\prime},c^{\prime})\subseteq G_{m}(a^{\prime},b^{\prime},c^{\prime})\subseteq G_{m}(a,b,c),
\]
 showing \eqref{eq:incluGB}. Now suppose $ p\in G_{m}(a^{\prime},b^{\prime},c^{\prime})\backslash B_{m}(a^{\prime},b^{\prime},c^{\prime})\not=\emptyset $. Since $p\in G_m(a',b',c')$, $p$ must divide two of $a'$, $b'$, and $c'$, and hence $ p^{2}\mid d(K)$, while $p\notin B_{m}(a',b',c')$ implies that $K_{p} $ is not split by $ \mathbb{H} $. But on the other hand, $ p\in G_{m}(a,b,c)\backslash B_{m}(a,b,c) $ from the containment \eqref{eq:incluGB} and so $ p^{2}\nmid d(K) $ or $ K_{p} $ is split by $ \mathbb{H} $ by the construction of $ K+u $, which is impossible. Thus $ G_{m}(a^{\prime},b^{\prime},c^{\prime})\backslash B_{m}(a^{\prime},b^{\prime},c^{\prime}) $ must be empty and hence $ G_{m}(a^{\prime},b^{\prime},c^{\prime})= B_{m}(a^{\prime},b^{\prime},c^{\prime})$.
\end{proof}

\begin{re}\label{re:wt_mgonalforms}
	Reordering the coefficients of the form $ \triangle_{m,(a^{\prime},b^{\prime},c^{\prime})} $ obtained by Lemma \ref{lem:wt_mgonalforms}, say $\triangle_{m,(a^{\prime\prime},b^{\prime\prime},c^{\prime\prime})}$, it is not difficult to see that $\triangle_{m,(a^{\prime\prime},b^{\prime\prime},c^{\prime\prime})}$ is still primitive and regular. Also, $ a^{\prime\prime}b^{\prime\prime}c^{\prime\prime}\mid abc $ and $ G_{m}(a^{\prime\prime},b^{\prime\prime},c^{\prime\prime})=B_{m}(a^{\prime\prime},b^{\prime\prime},c^{\prime\prime}) $. Hence we may require $ a^{\prime}\le b^{\prime}\le c^{\prime} $ in Lemma \ref{lem:wt_mgonalforms} further.
\end{re}

\subsection{Representation by $ \varphi_{m,(a,b,c)} $}
Based on the study of Dickson \cite{dickson_modern_1939}, Jones \cite{jones_new_1931} and Chan and B.-K. Oh \cite{chan_representations_2013}, we build sufficient conditions for a positive integer to be represented by $ \varphi_{m,(a,b,c)} $ over $ \mathbb{Z}_{p} $. Note that for any $ n\in\mathbb{N}$, we have $n \mathop{\rightarrow}\limits_{\mathbb{R}} \varphi_{m,(a,b,c)} $ if $ a,b$ and $c$ are positive integers, so we may suppose $ p\not=\infty $. We require two well-known lemmas \cite[Theorem 1 and Theorem 3, p.\hskip 0.1cm 41--42]{borevich_number_1986} in order to determine necessary conditions for solvability over $\Z_p$ to occur. 


\begin{lem}\label{lem:congeq_padic}
	Let $ F(x_{1},\cdots,x_{\ell})\in\mathbb{Z}[x_{1},\cdots,x_{\ell}] $. Then $ F(x_{1},\cdots,x_{\ell})\equiv 0\pmod{p^{t}} $ is solvable for all $ t\ge 1 $ if and only if the equation $ F(x_{1},\cdots,x_{\ell})=0 $ is solvable in $ \mathbb{Z}_{p} $.
\end{lem}

\begin{lem}\label{lem:hensel}
		Let $ F(x_{1},\cdots,x_{\ell})\in\mathbb{Z}[x_{1},\cdots,x_{\ell}] $. If $ \omega_{1},\ldots,\omega_{\ell}\in\mathbb{Z}_{p} $ is a solution of the following system of congruences
		\begin{align*}
		 F(\omega_{1},\cdots,\omega_{\ell})& \equiv 0 \pmod{p^{2t+1}}\,, \\
		F_{x_{i}} (\omega_{1},\cdots,\omega_{\ell})& \equiv 0 \pmod{p^{t}}\,,  \\
		 F_{x_{i}} (\omega_{1},\cdots,\omega_{\ell})& \not\equiv 0 \pmod{p^{t+1}}\,,  		 
		\end{align*}
		for some $ i $ ($ 1\le i\le \ell $), where $ t $ is a nonnegative integer and $F_{x}:=\frac{\partial F}{\partial x}$ denotes the derivative with respect to $x$, then the equation $ F(x_{1},\cdots,x_{\ell})=0 $ is solvable in $ \mathbb{Z}_{p} $.
\end{lem}

\begin{prop}\label{prop311}
	Let $ m, a,b,c,n $ be positive integers and $ p $ be prime.
\noindent

\noindent
\begin{enumerate}[leftmargin=*,align=left,label={\rm(\roman*).}]
\item Assume $\gcd(a,b,c)=1 $. If $ p\in P(2^{1-\delta}(m-2)) $, then $8(m-2)n+(m-4)^{2}(a+b+c) \mathop{\rightarrow}\limits_{\mathbb{Z}_{p}} \varphi_{m,(a,b,c)}$; If $ \ord_{2}(m)\ge 2 $ and $ n\equiv  \lfloor2/\small \ord_{2}(m)\rfloor(a+b+c) \pmod{8} $, then $8(m-2)n+(m-4)^{2}(a+b+c) \mathop{\rightarrow}\limits_{\mathbb{Z}_{2}}\varphi_{m,(a,b,c)}$.
\textsc{}\item If $ p\notin P(2(m-2))  $ and $ p\nmid abc $, then  $n \mathop{\rightarrow}\limits_{\mathbb{Z}_{p}} \varphi_{m,(a,b,c)}$.
\item If $p\notin P(2(m-2))$, $ p\mid c $ but $ p\nmid abn $, then $n \mathop{\rightarrow}\limits_{\mathbb{Z}_{p}} \varphi_{m,(a,b,c)}$.
\item If $p\notin P(2(m-2)) $, $ p\mid c $, $ p\mid n $, but $ p\nmid ab $, and if $ p\notin\mathbb{P}(-4ab) $,  then $n \mathop{\rightarrow}\limits_{\mathbb{Z}_{p}} \varphi_{m,(a,b,c)}$.
\end{enumerate}
\end{prop}
\begin{re}\label{re4}
	For $ p\in P_{m}(a,b,c) $, if $ p\nmid n $, then $ n \mathop{\rightarrow}\limits_{\mathbb{Z}_{p}} \varphi_{m,(a,b,c)} $ by Proposition \ref{prop311} {\rm (ii)} and {\rm (iii)}.
\end{re}

\begin{proof}
\noindent

\noindent
	{\rm (i).} 	Let $ p\in P(2(m-2)) $. Since $ \gcd(a,b,c)=1 $, we may assume without loss of generality that $ p\nmid a $. We split into cases based on $\ord_2(m)$. 

We first consider the case $\ord_{2}(m)=0$. Define the polynomial 
\[ 	F(x,y,z):=2\triangle_{m,(a,b,c)}(x,y,z)-2n \]
	in $ \mathbb{Z}[x,y,z] $ and then $ F_{x}(x,y,z)=2a(m-2)x-a(m-4) $. Also, $ \gcd(m-2,m-4)=1 $ and $ 2\nmid (m-2)(m-4) $. Take
	\begin{equation*}
	(x_{0},y_{0},z_{0})= 
	\begin{cases}
	(1,0,0) &  \mbox{if $ p=2 $},\\
	(-2n(a(m-4))^{-1},0,0) &  \mbox{if $ p\mid P_{m-2}$,}
	\end{cases}	
	\end{equation*}
where the inverse is taken in $\Z_p$. One can compute $ F(x_{0},y_{0},z_{0})\equiv 0\pmod{p} $ but $ F_{x}(x_{0},y_{0},z_{0})\equiv -a(m-4)\not\equiv 0 \pmod{p} $. Hence $ F(x,y,z)=0 $ is solvable in $ \mathbb{Z}_{p} $ by Lemma \ref{lem:hensel}, and thus there also exists a solution over $\Z_p$ to the equation $ \triangle_{m,(a,b,c)}(x,y,z)=n $.
	
We next assume that $\ord_2(m)\ge1$. In this case, we define the polynomial 
	\[ \widetilde{F}(x,y,z):=\triangle_{m,(a,b,c)}(x,y,z)-n \] 
in $ \mathbb{Z}[x,y,z] $. Then $ 	\widetilde{F}_{x}(x,y,z)=a(2^{t}m^{\prime}-2)x-a(2^{t-1}m^{\prime}-2) $, where $m'=m/2^{t}$ with $t\geq 1$ and $2\nmid m'$. Take
	\begin{equation*}
	(x_{0},y_{0},z_{0})= 
	\begin{cases}
	(1,0,0) &  \text{if $ t=1 $, $ p=2 $ and $2\nmid n$},\\
	(0,0,0) &  \text{if $ t=1 $, $ p=2 $ and $2\mid n$},\\
	(-n(a(2^{t-1}m^{\prime}-2))^{-1},0,0) &  \text{if $ p\mid P_{m-2}$.}
	\end{cases}	
	\end{equation*}
	One can see that $ \widetilde{F}(x_{0},y_{0},z_{0})\equiv 0\pmod{p} $ while $\widetilde{F}_{x}(x_{0},y_{0},z_{0})\equiv -2a(2^{t-2}m^{\prime}-1)\not\equiv 0 \pmod{p}  $.  Hence $ \triangle_{m,(a,b,c)}(x,y,z)=n $ is solvable in $ \mathbb{Z}_{p} $ for $ p\mid P_{m-2} $ and $ p=2 $, $ \ord_{2}(m)=1 $ by Lemma \ref{lem:hensel}.
	
We finally consider the case $\ord_2(m)\geq 2$ and $ p=2 $. If $ \ord_{2}(m)=2 $ and $ n\equiv a+b+c\pmod{8} $, then one can put $ (x_{1},y_{1},z_{1})=(1,1,1) $ and check that
	\begin{align*}
	\widetilde{F}(x_{1},y_{1},z_{1})&=(a+b+c)(2m^{\prime}-1)-2(m^{\prime}-1)(a+b+c)-n\\
	&=a+b+c-n\equiv 0 \pmod{8},
	\end{align*}
while $ \widetilde{F}_{x}(x_{1},y_{1},z_{1})=2^{t-1}m^{\prime}a=2m^{\prime}a\equiv 2\pmod{4} $.  Hence $ \triangle_{m,(a,b,c)}(x,y,z)=n $ is solvable in $ \mathbb{Z}_{2} $ by Lemma \ref{lem:hensel} (taking $t=1$). If $ \ord_{2}(m)>2 $ and $ n\equiv 0\pmod{8} $, then $ 2\nmid 2^{t-2}m^{\prime}-1 $. Put $ (x_{2},y_{2},z_{2})=(0,0,0) $ in this case. One can see that $\widetilde{F}(x_{2},y_{2},z_{2})=-n\equiv 0 \pmod{8}$ while $ \widetilde{F}_{x}(x_{2},y_{2},z_{2})=-2a(2^{t-2}m^{\prime}-1)\equiv 2\pmod{4} $.  Hence $ \triangle_{m,(a,b,c)}(x,y,z)=n $ is solvable in $ \mathbb{Z}_{2} $ by Lemma \ref{lem:hensel}.
\vspace{.05in}
	
\noindent
	{\rm (ii)--(iv).} Assume $ p\notin P(2(m-2)) $. Since the linear map $x\mapsto 2(m-2)x-(m-4)$ is a bijection in $\Z_p$, the statements {\rm (ii)-(iv)} follow immediately from Dickson's results \cite[Lemma 3-5, p.\hskip0.1cm 107]{dickson_modern_1939} and Lemma \ref{lem:congeq_padic}.
\end{proof}

For given positive integers $ a,b $ and $ c $ with $ \gcd(a,b,c)=1 $, we may assume $ G_{m}(a,b,c)=B_{m}(a,b,c) $ by Lemma \ref{lem:wt_mgonalforms} and $ a\le b\le c $ by Remark \ref{re:wt_mgonalforms}. Then $8(m-2)n+(m-4)^{2}(a+b+c) \mathop{\rightarrow}\limits_{\mathbb{Z}_{p}} \varphi_{m,(a,b,c)}$ for $ p\in P(2^{1-\delta}(m-2)) $ and $ p\nmid abc $ by Proposition \ref{prop311} {\rm (i)} and {\rm (ii)}. For $ p\in G_{m}(a,b,c)=B_{m}(a,b,c) $, $8(m-2)n+(m-4)^{2}(a+b+c) \mathop{\rightarrow}\limits_{\mathbb{Z}_{p}} \varphi_{m,(a,b,c)}$ by Remark \ref{re:split}. For $ p\mid abc $, without loss of generality, let $ p\mid a$. If $ (-4bc/p)=1 $, then $8(m-2)n+(m-4)^{2}(a+b+c) \mathop{\rightarrow}\limits_{\mathbb{Z}_{p}} \varphi_{m,(a,b,c)}$ by Proposition \ref{prop311} {\rm (iii)} if $ p\nmid 8(m-2)n+(m-4)^{2}(a+b+c) $ and by Proposition \ref{prop311} {\rm (iv)} if $ p\mid 8(m-2)n+(m-4)^{2}(a+b+c) $. Therefore, to check whether $ 8(m-2)n+(m-4)^{2}(a+b+c) $ is locally represented by $ \varphi_{m,(a,b,c)} $ or not, it is sufficient to consider the local representation over $ \mathbb{Z}_{p} $ for $ p\in P_{m}(a,b,c)\cup \{2\}^{\delta}$. In other words, to show Theorem \ref{thm:noregularmgonalforms}, it is enough to consider the case of $ G_{m}(a,b,c)=B_{m}(a,b,c) $.

\section{Bounding the coefficients $ a $, $ b $ and $ c $}\label{sec:bound_abc}

For each fixed integer $ m>3 $, we always assume that $ q_{0} $ is the smallest prime in the set $\mathbb{P}(-4ab)\backslash (P(m-2)\cup P_{m}(c,ab))$ (the existence follows from Lemma \ref{lem:tool2} {\rm (i)} below) and denote by $ \{q_{i}\}_{ab,m} $ ($ i=1,2,\ldots $) the sequence of all primes in $\mathbb{P}(-4ab)\backslash P(q_{0}(m-2)) $ in ascending order for brevity. The following useful proposition may be found in \cite[Lemma 3.5]{kim_2-universal_1999}.

\begin{prop}[B. M. Kim, M.-H. Kim, and B.-K. Oh \cite{kim_2-universal_1999}]\label{prop:tool1}
	Let $ T $ be a finite set of primes. Set $ P:=\prod_{p\in T}p $ and let $ \ell $ be an integer relatively prime to $P$. Then for any integer $u$, the number of integers in the set
	\begin{align*}
	\{u,\ell+u,\cdots,\ell(n-1)+u\}
	\end{align*}
	that are relatively prime to $P$ is at least
	$n\phi(P)/P-2^{\omega(P)}+1.$
\end{prop}

We next give an upper bound on the product $ab$ for a regular ternary $m$-gonal form by using Proposition \ref{prop:tool1}. 
\begin{lem}\label{lem:bound_ab}
	Let $ a\le b\le c$ be positive integers for which $ \gcd(a,b,c)=1 $ and $G_m(a,b,c)=B_{m}(a,b,c)$. If $ \triangle_{m,(a,b,c)} $ is regular, then 
	\begin{align*}
	a \le 8\rho(P_{abc}) \hskip0.5cm \mbox{and} \hskip 0.5cm b \le 64\cdot 11 P_{abc}\rho(P_{abc})/\phi(P_{abc}).
	\end{align*}
	Hence 
	$ab<C_{2}P_{abc}\rho(P_{abc})^{2}/\phi(P_{abc})$,
where $ C_{2}:=2^{9}\cdot 11 $ is a constant.
\end{lem}
\begin{re}\label{rem:relprime}
At first glance, it is not obvious whether the right-hand sides of the inequalities in Lemma \ref{lem:bound_ab} grow faster or slower than the left-hand sides. However, since $\frac{n}{\phi(n)}=O(n^{\varepsilon})$, the right-hand side grows like $P_{abc}^{\varepsilon}\ll (abc)^{\varepsilon}$. Thus if $c$ may be bounded as a function of $a$ and $b$ slower than $(ab)^{1/\varepsilon}$, then such a bound may be combined with Lemma \ref{lem:bound_ab} to obtain a restriction on the possible choices of $a$, $b$, and $c$ for which the form $\triangle_{m,(a,b,c)}$ may be regular.
\end{re}
\begin{proof}	
	If $ \ord_{2}(m)<2 $, we take $ w_{0}=1 $; if $ \ord_{2}(m)\ge 2 $, we choose $w_0$ such that $0< w_{0}\le 8 $ and $w_{0}\equiv  \lfloor2/\small \ord_{2}(m)\rfloor(a+b+c)  \pmod{8}$. Clearly, for every $v\in\N$, if $ \ord_{2}(m)\ge 2 $, then the integer $n=8^{\delta}v+w_{0}$ is congruent to $\lfloor2/\small \ord_{2}(m)\rfloor(a+b+c)  $ modulo $ 8 $. 

We next construct a pair of integers $N_0$ and $\bar{N}_{0}$ which are locally represented by $\varphi_{m,(a,b,c)}$ and then use the regularity to obtain upper bounds for $a$ and $b$. We do so in a series of steps (a)--(c) below, first constructing them in (a), showing that they are locally represented in (b), and finally obtaining the bounds for $a$ and $b$ in (c).

	\noindent (a) We first construct the integers $ N_{0} $ and $ \bar{N}_{0} $ such that $ \gcd(N_{0},P_{abc})=1=\gcd(\bar{N}_{0},P_{abc})$, $ N_{0}\equiv (m-4)^{2}(a+b+c)\pmod{8(m-2)}$ and there exists some positive integer $ \bar{n_{0}} $ such that $ \bar{N_{0}}=8(m-2)\bar{n_{0}}+(m-4)^{2}(a+b+c) $ and $ \bar{n}_{0}\mathop{\not\rightarrow}\limits_{\mathbb{Z}}\triangle_{m,a}  $.
	
	Write $ u=8(m-2)w_{0}+(m-4)^{2}(a+b+c) $. Since $ \gcd(2(m-2),P_{abc})=1 $, when $ v= \lfloor\rho(P_{abc})\rfloor $, $ v\phi(P_{abc})/P_{abc}-2^{\omega(P_{abc})}+1>0 $. By Proposition \ref{prop:tool1}, there exists at least one integer $ 0\le v_{0}\le v-1 $ such that $N_{0}:=8^{\delta+1}(m-2)v_{0}+u$ is relatively prime to $P_{abc}$ and $ N_{0}\equiv u\equiv (m-4)^{2}(a+b+c) \pmod{8(m-2)} $. 
	
We now construct $ \bar{N}_{0} $. We claim that for $ v\ge 2 $, there are at most $ 2\sqrt{2\cdot 8^{\delta}(v+1)+1/4} $ integers between $1$ and $v$ represented by $ (\triangle_{m,a}(x)-w_{0})/8^{\delta} $. Indeed, solving the inequality
	\begin{align*}
	a\hskip 0.1cm \dfrac{(m-2)x^{2}-(m-4)x}{2\cdot 8^{\delta}}-\dfrac{w_{0}}{8^{\delta}}\le v
	\end{align*}
	for $ x $, we see that $ x_{-}\le x\le x_{+} $, where
	\begin{align*}
	x_{\pm}=\pm\sqrt{\dfrac{2(8^{\delta}v+w_{0})}{a(m-2)}+\left(\dfrac{m-4}{2(m-2)}\right)^{2}}+\dfrac{m-4}{2(m-2)}.
	\end{align*}
	Since $ a\ge 1 $, $ m>3 $ and $ w_{0}\le 8^{\delta} $, we have
	\begin{align*}
	2\sqrt{\dfrac{2(8^{\delta}v+w_{0})}{a(m-2)}+\left(\dfrac{m-4}{2(m-2)}\right)^{2}}<2\sqrt{2\cdot 8^{\delta}(v+1)+1/4},
	\end{align*}
yielding the claim. Taking $v=\lfloor 11\cdot 8^{\delta} P_{abc}\rho(P_{abc})/\phi(P_{abc})\rfloor$, one can compute
	\begin{align*}
	v\phi(P_{abc})/P_{abc}-2^{\omega(P_{abc})}+1&>2\sqrt{2\cdot 8^{\delta}(v+1)+1/4}.
	\end{align*} 
	Since $ \gcd(2(m-2),P_{abc})=1 $, Proposition \ref{prop:tool1} implies that there exists an integer $ 0\le \bar{v}_{0}\le v-1 $ for which $ \bar{N}_{0}:=8^{\delta+1}(m-2)\bar{v}_{0}+u $ is relatively prime to $P_{abc}$. Furthermore, $ \bar{v}_{0}$ is not represented by $ (\triangle_{m,a}(x)-w_{0})/8^{\delta} $;
	namely, $ \bar{n_{0}}:=8^{\delta}\bar{v}_{0}+w_{0} $ is not represented by $ \triangle_{m,a}(x) $ as desired.

	\noindent (b) For each prime $ p $, we have $ N_{0}\mathop{\rightarrow}\limits_{\mathbb {Z}_{p}}\varphi_{m,(a,b,c)} $ and $ \bar{N}_{0}\mathop{\rightarrow}\limits_{\mathbb {Z}_{p}}\varphi_{m,(a,b,c)} $.

	By the construction in (a), we see that $ N_{0} $ can be rewritten as $8(m-2)n_{0}+(m-4)^{2}(a+b+c)$,
	where $ n_{0}=8^{\delta}v_{0}+w_{0} $ is a positive integer.
	 Clearly, when $ \ord_{2}(m)\le 1 $, by the first part of Proposition \ref{prop311} {\rm (i)}, $N_{0}\mathop{\rightarrow}\limits_{\mathbb{Z}_{2}}\varphi_{m,(a,b,c)}$; when  $ \ord_{2}(m)\ge2 $, $ n_{0}\equiv \lfloor2/\small \ord_{2}(m)\rfloor(a+b+c) \pmod{8}  $ and by the second part of Proposition \ref{prop311} {\rm (i)} we conclude that $N_{0}\mathop{\rightarrow}\limits_{\mathbb{Z}_{2}}\varphi_{m,(a,b,c)}$.
	 For $ p\in P_{m}(a,b,c) $, since $ \gcd(N_{0},P_{abc})=1 $, by Remark \ref{re4} we further see that
	 $N_{0}\mathop{\rightarrow}\limits_{\mathbb {Z}_{p}}\varphi_{m,(a,b,c)}$. Hence $N_{0}\mathop{\rightarrow}\limits_{\mathbb{Z}_{p}}\varphi_{m,(a,b,c)}$ for each prime $ p $. Note that $ \bar{N}_{0} $ can be also rewritten as $8(m-2)\bar{n}_{0}+(m-4)^{2}(a+b+c)$, where $ \bar{n}_{0}=8^{\delta}\bar{v}_{0}+w_{0}>0 $. Repeating the above argument, we deduce that $\bar{N}_{0}\mathop{\rightarrow}\limits_{\mathbb{Z}_{p}}\varphi_{m,(a,b,c)}$ for each prime $ p $.
	 
	\noindent (c) We finally use $N_0$ and $\bar{N}_{0}$ to bound $ a $ and $ b $.
	 
	From (b), we see that $ N_{0} $ is locally represented by $ \varphi_{m,(a,b,c)}$. Since $ \varphi_{m,(a,b,c)} $ is regular,  $N_{0}$ is globally represented by $ \varphi_{m,(a,b,c)}$. It follows that $ n_{0}\mathop{\rightarrow}\limits_{\mathbb{Z}}\triangle_{m,(a,b,c)}$ from Remark \ref{re1}. Hence 
	\[ a\le n_{0}= 8^{\delta}v_{0}+w_{0}\le 8^{\delta}(\rho(P_{abc})-1)+8^{\delta}\le 8\rho(P_{abc}). \]
	From the construction in (a), we see that $ \bar{N}_{0} $ can be written as $8(m-2)\bar{n}_{0}+(m-4)^{2}(a+b+c)$, where $ \bar{n}_{0}=8^{\delta}\bar{v}_{0}+w_{0} $ is not represented by $ \triangle_{m,a} $. By (b), $\bar{N}_{0}$ is also locally represented by $ \varphi_{m,(a,b,c)}$ and so also globally represented from the regularity of $ \varphi_{m,(a,b,c)} $. This implies $ \bar{n}_{0}\mathop{\rightarrow}\limits_{\mathbb{Z}}\triangle_{m,(a,b,c)} $. 
Since $ \bar{n}_{0}\mathop{\not\rightarrow}\limits_{\mathbb{Z}}\triangle_{m,a} $, it must be the case that
	\begin{align*}
	b\le \bar{n}_{0}= 8^{\delta}\bar{v}_{0}+w_{0}&<8^{\delta}(11\cdot 8^{\delta}P_{abc}\rho(P_{abc})/\phi(P_{abc})-1)+8^{\delta} \\ 
&\leq 64\cdot 11P_{abc}\rho(P_{abc})/ \phi(P_{abc}).\qedhere
	\end{align*} 
\end{proof}

\begin{lem}\label{lem:tool2}
	Let $ a,b,c $ be positive integers. Set $ C_{1}:=2^{9/5}C_{0}^{6/5}$, where $ C_{0} $ is the constant defined as in Lemma \ref{lem:newprimeanalytic}. Fix an integer $ m\geq 4 $. Then the following hold.
	\begin{enumerate}[leftmargin=*,align=left,label={\rm(\roman*).}]
	\item We have $ q_{0}<4^{2/3}C_{0}(P_{m-2}P_{c})^{1/6}(ab)^{2/3}$.
	
  	\item We have  
     $q_{1}<C_{1}P_{m-2}^{1/6}(q_{0}^{3}K(a,b,c))^{1/5}(ab)^{4/5}.$
\item
	Assume that $q_{i_{0}+1} $ is the least prime in $ \{q_{i} \}_{ab,m} $ greater than 
	\[ 
	C_{1}P_{m-2}^{1/6}(q_{0}^{3}K(a,b,c))^{1/5}(ab)^{4/5}. \]
	 Then the inequality 
	\begin{equation}\label{eq34}
	\begin{aligned}
	K(a,b,c)q_{0}^{2}q_{i+1}<(m-3)q_{1}q_{2}\cdots q_{i}
	\end{aligned}
	\end{equation}
	 holds for $ i\ge i_{0} $.
\end{enumerate}
\end{lem}
\begin{proof}
	Let $ t $ be the least positive integer such that $ P_{m-2}/t $ is prime to $ 4ab $. Then $ P_{m-2}=tu $, where $ u\ge 1 $ and $ \gcd(2ab,u)=1 $.
	
	{\rm (i).} Take $ D=-4ab $ and $ M=P_{c}u$ in Lemma \ref{lem:newprimeanalytic}.  We see that there exists some prime $ q\in \mathbb{P}(-4ab)\backslash (P(m-2)\cup P_{m}(c,ab))  $ such that  
	\begin{align*}
	q_{0}\le q<C_{0}(4ab)^{2/3}(P_{c}u)^{1/6}\le 4^{2/3}C_{0}(P_{m-2}P_{c})^{1/6}(ab)^{2/3}.
	\end{align*}
	
	{\rm (ii).} Taking $ D=-4ab $ and $ M=q_{0}u$ in Lemma \ref{lem:newprimeanalytic}, we see that there exists some prime $ q\in \mathbb{P}(-4ab)\backslash P(q_{0}(m-2)) $ such that  
	\begin{align*}
	q<C_{0}(4ab)^{2/3}(q_{0}u)^{1/6}&<4^{2/3}C_{0}(q_{0}P_{m-2})^{1/6}(ab)^{2/3}\\
	&<C_{1}P_{m-2}^{1/6}(q_{0}^{3}K(a,b,c))^{1/5}(ab)^{4/5}.
	\end{align*}

	{\rm (iii).} Now suppose that $ q_{1}q_{2}\cdots q_{j}\le K(a,b,c)q_{0}^{2}q_{j+1}/(m-3)$ for some $ j\ge i_{0} $.  Taking $ D=-4ab $ and $ M=q_{1}q_{2}\cdots q_{j}q_{0}u $ in Lemma \ref{lem:newprimeanalytic}, one deduces that there exists some prime $ q^{\prime}\in \mathbb{P}(-4ab)\backslash P(q_{0}(m-2)) $ different from $ q_{1},q_{2},\cdots,q_{j} $ such that
	\begin{align*}
	q_{j+1}\le q^{\prime} &<C_{0}(4ab)^{2/3}(q_{1}q_{2}\cdots q_{j}q_{0}u)^{1/6}\\
	&<C_{0}4^{2/3} (q_{0}^{3}P_{m-2}K(a,b,c)/(m-3))^{1/6}(ab)^{2/3}q_{j+1}^{1/6}\\
	&\le C_{0}4^{2/3}(2q_{0}^{3}K(a,b,c))^{1/6}(ab)^{2/3}q_{j+1}^{1/6},
	\end{align*} 
where in the last line we bounded $P_{m-2}/(m-3)\leq (m-2)/(m-3)\leq 2$ for $m\geq 4$. It follows that
	\begin{align*}
	q_{j+1}&<C_{0}^{6/5}4^{4/5}  (2q_{0}^{3}K(a,b,c))^{1/5}(ab)^{4/5}\\
	&\leq C_{1}P_{m-2}^{1/6}(q_{0}^{3}K(a,b,c))^{1/5}(ab)^{4/5}<q_{i_{0}+1},  
	\end{align*}
which contradicts the assumption that $j\ge i_{0}$.
\end{proof}

\begin{lem}\label{lem:Ni}
	Let $ a,b,c $ be positive integers. Then for each $ q_{i} $ in the prime sequence $ \{q_{i}\}_{ab,m} $ ($ i=1,2,\cdots $), there exists some $ N_{i} $ such that $ \ord_{q_{i}}N_{i}=1 $,  $ N_{i}\equiv (m-4)^{2}(a+b) \pmod{8(m-2)}$, $ N_{i}\equiv 8(m-2)c+q_{0} \pmod{q_{0}^{2}} $ and $ \gcd(N_{i}/q_{i},P_{c})=1 $.
	
	Hence there exists a positive integer $ n_{i} $ such that $ N_{i}=8(m-2)n_{i}+(m-4)^{2}(a+b) $. Also, $ n_{i}\le K(a,b,c)q_{0}^{2}q_{i} $. If moreover $ \ord_{2}(m)\ge 2 $, then $ n_{i}\equiv \lfloor2/\small \ord_{2}(m)\rfloor (a+b+c) \pmod{8}$.
\end{lem}
\begin{proof}
	Write $ P_{ab}=P_{ab}^{\prime}s $, where $ \gcd(s, P_{ab}^{\prime})=1 $. For $ i=1,2,\cdots $, observe that 
$\gcd(sq_{i},2(m-2)q_{0})=1$ from  $2\nmid P_{ab}$, and $ \gcd(q_{i},2(m-2)q_{0}ab)=1 $. Also, $\gcd(q_{0},8(m-2))=1 $ and moreover, when $ \ord_{2}(m)\ge 2 $, $
	\gcd(8(m-2),128)=16\mid 8(m-2) \lfloor 2/\ord_{2}(m)\rfloor(a+b+c).$
	By the Chinese remainder theorem, the system of congruences 
	\begin{equation*}
	\left\{
	\begin{aligned}
	sq_{i}u&\equiv 8(m-2)c+q_{0} \pmod{q_{0}^{2}}  \\
	sq_{i}u&\equiv (m-4)^{2}(a+b) \pmod{8(m-2)} \\
	sq_{i}u&\equiv 8(m-2)\lfloor 2/\ord_{2}(m)\rfloor(a+b+c)+(m-4)^{2}(a+b) \pmod{128^{\delta}} \\
	\end{aligned}
	\right.
	\end{equation*}
	is solvable (in terms of $ u $) for each $ m>3 $. Then we take its solution, say $ u_{i} $, in the range 
\[
 (m-4)^{2}(a+b)/(sq_{i})< u_{i}\le  8^{\delta+1}(m-2)q_{0}^{2}+(m-4)^{2}(a+b)/(sq_{i}).
\]
Since $8^{\delta+1}(m-2)q_{0}^{2}v+u_{i} $ is also a solution for any $ v\in\mathbb{N} $ (note that $ 128\mid 8^{\delta+1}(m-2) $ when $ \ord_{2}(m)\ge 2 $) and $ \gcd(2(m-2)q_{0}^{2},q_{i}P_{c}P_{ab}^{\prime})=1 $, choosing $ v=\lfloor\rho(q_{i}P_{c}P_{ab}^{\prime})\rfloor$, we see by Proposition \ref{prop:tool1} that there exists at least one integer $ 0\le v_{i} \le v-1 $ for which $ w_{i}:=8^{\delta+1}(m-2)q_{0}^{2}v_{i}+u_{i} $ satisfies $ \gcd(w_{i},q_{i}P_{c}P_{ab}^{\prime})=1 $. Take $ N_{i}:=sq_{i}w_{i} $. Since $ \gcd(sw_{i},q_{i})=1 $, $\ord_{q_{i}}N_{i}=1$. Also, since $ \gcd(P_{ab}w_{i}, P_{c})=1 $, it follows that $ \gcd(N_{i}/q_{i}, P_{c})=\gcd(sw_{i},P_{c})=1 $. Moreover,
	\begin{align}
\label{eqn:Nibound}
	(m-4)^{2}(a+b)< N_{i}&=sq_{i}(8^{\delta+1}(m-2)q_{0}^{2}v_{i}+u_{i})\\
\nonumber	&\le sq_{i}\Big(8^{\delta+1}(m-2)q_{0}^{2}(\rho(q_{i}P_{c}P_{ab}^{\prime})-1)+\\
\nonumber	&\hskip1.25cm 8^{\delta+1}(m-2)q_{0}^{2}+(m-4)^{2}(a+b)/(sq_{i})\Big)  \\
\nonumber	&\le sq_{i}8^{\delta+1}(m-2)q_{0}^{2}\rho(q_{i}P_{c}P_{ab}^{\prime})+(m-4)^{2}(a+b)  \\
\nonumber	&\le 24\cdot 8^{\delta}(m-2)q_{0}^{2}P_{ab}\rho(P_{abc})q_{i}+(m-4)^{2}(a+b),
	\end{align}
	as $ \rho(q_{i})\le 3 $, $ s\le P_{ab} $ and $ P_{ab}^{\prime}\mid P_{ab} $. 
	
	For the second part, from the last two congruences, we have 
\[
N_{i}=sq_{i}w_{i}=8(m-2)8^{\delta}t_{i}+(m-4)^{2}(a+b)
\]
 for some positive integer $ t_{i} $, as $ N_{i}>(m-4)^{2}(a+b) $ by construction. Take $ n_{i}:=8^{\delta}t_{i} $. Then $ n_{i}>0 $. Since $N_{i}\le  24\cdot 8^{\delta}(m-2)q_{0}^{2}P_{ab}\rho(P_{abc})q_{i}+(m-4)^{2}(a+b)$ by \eqref{eqn:Nibound}, we have 
	\begin{align*}
	n_{i}\le 3\cdot 8^{\delta}P_{ab}\rho(P_{abc})q_{0}^{2}q_{i}\le K(a,b,c)q_{0}^{2}q_{i}.
	\end{align*}
 When $ \ord_{2}(m)\ge 2 $, we also have 
	\[  N_{i} \equiv 8(m-2) \lfloor 2/\ord_{2}(m)\rfloor(a+b+c)+(m-4)^{2}(a+b)\pmod{128}, \]
	which implies $ n_{i}\equiv  \lfloor 2/\ord_{2}(m)\rfloor (a+b+c) \pmod{8}$.
\end{proof}

\begin{lem}\label{lem:Ni_localrepre}
	Let $ a,b,c$ be positive integers with $ a\le b\le c $ and $ \gcd(a,b,c)=1 $. Let $ N_{i} $ be the integers as defined in Lemma \ref{lem:Ni}, $ i=1,2,\cdots $. Then the following hold.  
	\begin{enumerate}[leftmargin=*,align=left,label={\rm(\roman*).}]
	\item  We have $ N_{i}\mathop{\not\rightarrow}\limits_{\mathbb{Z}}\varphi_{m,(a,b)}$; if $ N_{i}-8(m-2)c>0 $, then $ N_{i}-8(m-2)c\mathop{\not\rightarrow}\limits_{\mathbb{Z}}\varphi_{m,(a,b)}$.
	
	\item  If $ p\in P_{m}(a,bc)\cup P_{m}(b,ac)\cup \{2\}^{\delta}$, then $ N_{i}+(m-4)^{2}c\mathop{\rightarrow}\limits_{\mathbb{Z}_{p}}\varphi_{m,(a,b,c)} $.
	
	\item If $ p\in P_{m}(c,ab)$ and $q_i\nmid c$, then $ N_{i}+(m-4)^{2}c\mathop{\rightarrow}\limits_{\mathbb{Z}_{p}}\varphi_{m,(a,b,c)} $.
	
	\item  If $ q_i\nmid c$ and $a$, $b$, and $c$ satisfy $G_m(a,b,c)=B_{m}(a,b,c)$, then $ N_{i}+(m-4)^{2}c\mathop{\rightarrow}\limits_{\mathbb{Z}_{p}}\varphi_{m,(a,b,c)} $ for each prime $ p $.
\end{enumerate}
\end{lem}
\begin{proof}
Let $ N_{i}=sq_{i}w_{i} $ be as constructed in Lemma \ref{lem:Ni} for $ i=1,2,\cdots $.
	
	{\rm (i).} If $ N_{i} $ is represented by $ \varphi_{m,(a,b)} $, then the equation $ N_{i}=ax^{2}+by^{2} $ is solvable (in $\mathbb{Z}$) and so is $ 4aN_{i}=x^{2}+4aby^{2} $. Hence $ x^{2}+4aby^{2}\equiv 0\pmod{q_{i}} $. But $   (-4ab/q_{i})=-1 $, which implies that $ x\equiv y\equiv 0\pmod{q_{i}} $. So $q_{i}^{2}\mid  x^{2}+4aby^{2}=4aN_{i}$
	and hence $ q_{i}^{2}\mid N_{i} $, which contradicts $ \ord_{q_{i}}N_{i}=1 $. 
	
	If $ N_{i}-8(m-2)c>0 $, then since $ N_{i}\equiv 8(m-2)c+q_{0} \pmod{q_{0}^{2}} $, we conclude that $ \ord_{q_{0}}(N_{i}-8(m-2)c)=1 $. If $N_i-8(m-2)c$ is represented by $ \varphi_{m,(a,b)} $ over $\Z$, then since $(-4ab/q_0)=-1$, we again conclude by a similar argument that $q_0^2\mid N_{i}-8(m-2)c $, yielding a contradiction.

	{\rm (ii).} Let $ p\in P_{m}(a,bc)\cup P_{m}(b,ac)$ be given, from which we conclude that $p\mid P_{ab} $ but $ p\nmid c $. We again write $P_{ab}=P_{ab}'s$, with $\gcd(P_{ab}',s)=1$, so that either $ p\mid s $ or $ p\mid P_{ab}^{\prime} $. If $ p\mid s $, then $ p\nmid m-4 $ and it follows that $ \gcd(N_{i}+(m-4)^{2}c,p)=\gcd((m-4)c,p)=1 $ and so $ N_{i}+(m-4)^{2}c\mathop{\rightarrow}\limits_{\mathbb{Z}_{p}}\varphi_{m,(a,b,c)}$ by Remark \ref{re4}. On the other hand, if $ p\mid P_{ab}^{\prime} $, then $ p\mid m-4 $ and $p\nmid s$ and hence  
\[
\gcd(N_{i}+(m-4)^{2}c,p)=\gcd(N_{i},p)=\gcd(sq_iw_i,p)=\gcd(q_iw_i,p).
\]
We then note that $ p\neq q_{i} $ because $ \gcd(q_{i},ab)=1 $ and $ \gcd(w_{i},P_{ab}^{\prime})=1 $, from which we conclude that $ \gcd(q_iw_{i},p)=1 $. So $ N_{i}+(m-4)^{2}c\mathop{\rightarrow}\limits_{\mathbb{Z}_{p}}\varphi_{m,(a,b,c)} $ by Remark \ref{re4}.

	Now consider $ p=2 $ and assume without loss of generality that $ \ord_{2}(m)\ge 2 $, since otherwise this case is covered above. Observe that $ N_{i}+(m-4)^{2}c=8(m-2)n_{i}+(m-4)^{2}(a+b+c)  $ and $ n_{i}\equiv  \lfloor 2/\ord_{2}(m)\rfloor (a+b+c) \pmod{8}$, where $ n_{i} $ is constructed as in the proof of Lemma \ref{lem:Ni_localrepre}. By the second part of Proposition \ref{prop311} {\rm (i)}, we have $ N_{i}+(m-4)^{2}c\mathop{\rightarrow}\limits_{\mathbb{Z}_{2}}\varphi_{m,(a,b,c)} $.

	{\rm (iii).} Let $ p\in P_{m}(c,ab) $ with $q_i\nmid c$ be given. Then $ p\mid P_{c} $ but $ p\nmid  ab $. Since $p\mid c$, $s\mid P_{ab}$, $\gcd(w_i,P_c)=1$, and $q_i\nmid c$ by construction (and hence $q_i\neq p$), we have  
\[
\gcd(N_{i}+(m-4)^{2}c,p)=\gcd(N_{i},p)=\gcd(sw_iq_i,p)=1.
\]
Therefore $ N_{i}+(m-4)^{2}c\mathop{\rightarrow}\limits_{\mathbb{Z}_{p}}\varphi_{m,(a,b,c)}$ by Remark \ref{re4}.
	
	{\rm (iv).} The statement follows immediately from parts {\rm (ii)} and {\rm (iii)}.
\end{proof}

Recall from Remark \ref{rem:relprime} that one obtains a bound for $a$ and $b$ in a regular $\triangle_{m,(a,b,c)}$ which is non-trivial when $a$, $b$, and $c$ satisfy $G_m(a,b,c)=B_{m}(a,b,c)$. It was then explained that obtaining a bound for $c$ in terms of $a$ and $b$ would lead to a bound for the possible choices of $a$, $b$, and $c$ for which $\triangle_{m,(a,b,c)}$ is regular. Following Dickson's proofs of \cite[Theorem 5, 6]{dickson_ternary_1926}, we deduce such a bound for $c$.

\begin{lem}\label{lem:bound_c}
	Let $ a\le b\le c $ be positive integers for which $\gcd(a,b,c)=1$ and $G_m(a,b,c)=B_{m}(a,b,c)$. If $ \triangle_{m,(a,b,c)} $ is regular, then 
	\begin{align*}
	(m-3)c\le C_{3}P_{m-2}^{3/5}P_{c}^{1/2}K(a,b,c)^{6/5}(ab)^{38/15}
	\end{align*}
	with the constant $ C_{3}:=4^{26/15}C_{0}^{13/5}C_{1} $, where $C_{0}$ and $C_{1} $ are defined as in Lemma \ref{lem:newprimeanalytic} and \ref{lem:tool2}, respectively.
\end{lem}
\begin{proof}
	Consider the prime sequence $ \{q_{i}\}_{ab,m} $. By Lemma \ref{lem:tool2} (ii), we see that
	\begin{align*}
	q_{1}<q_{2}<\cdots<q_{i_{0}}<C_{1}P_{m-2}^{1/6}(q_{0}^{3}K(a,b,c))^{1/5}(ab)^{4/5}<q_{i_{0}+1}<\cdots,
	\end{align*}
	for some $ i_{0}\ge 1 $. For each $ i $, we take $ N_{i}=sq_{i}w_{i}=8(m-2)n_{i}+(m-4)^{2}(a+b) $ as constructed in Lemma \ref{lem:Ni}.
	If $ q_j\nmid c $ for some $ 1\le j\le i_{0} $, then  $N_{j}+(m-4)^{2}c\mathop{\rightarrow}\limits_{\mathbb{Z}_{p}}\varphi_{m,(a,b,c)}$ for every prime $ p $ by Lemma \ref{lem:Ni_localrepre} {\rm (iv)}. But then $ n_{j}\mathop{\rightarrow}\limits_{\mathbb{Z}_{p}}\triangle_{m,(a,b,c)} $ for each prime $ p $ and $ \triangle_{m,(a,b,c)} $ is regular so $ n_{j}\mathop{\rightarrow}\limits_{\mathbb{Z}}\triangle_{m,(a,b,c)} $. Hence $N_{j}+(m-4)^{2}c\mathop{\rightarrow}\limits_{\mathbb{Z}}\varphi_{m,(a,b,c)} $. Namely, 
\begin{equation}\label{eqn:Njz0}
	N_{j}=\varphi_{m,(a,b)}(x_{0},y_{0})+4c(m-2)z_{0}((m-2)z_{0}-(m-4))
\end{equation}
	for some $ x_{0},y_{0},z_{0}\in\mathbb{Z} $. Since $ N_{j}\mathop{\not\rightarrow}\limits_{\mathbb{Z}}\varphi_{m,(a,b)} $ by Lemma \ref{lem:Ni_localrepre} {\rm (i)}, it follows that $ z_{0}\neq 0$. Moreover, since \eqref{eqn:Njz0} is increasing as a function of $z_0$ for $z_0>0$ and decreasing for $z_0<0$, we have $
N_j-8(m-2)c\geq \varphi_{m,(a,b)}(x_0,y_0)$.
Since $N_j-8(m-2)c\equiv q_0\pmod{q_0^2}$ from Lemma \ref{lem:Ni}, we have $ N_{j}-8(m-2)c>0$ and so $ N_{j}-8(m-2)c\mathop{\not\rightarrow}\limits_{\mathbb{Z}}\varphi_{m,(a,b)}$ by Lemma \ref{lem:Ni_localrepre} {\rm (i)}, which in turn implies that $ z_{0}\neq 1$. Hence $ 	N_{j}\ge \varphi_{m,(a,b)}(x_{0},y_{0})+8c(m-2)(m-3) $; that is
	$n_{j}\ge \triangle_{m,(a,b)}(x_{0},y_{0})+c(m-3)$ and so, using the fact that $n_j\leq K(a,b,c)q_0^2q_j$ by Lemma \ref{lem:Ni} and $q_j\leq q_{i_0}$,
	\begin{align*}
	(m-3)c\le n_{j}\le K(a,b,c)q_{0}^{2}q_{j}&\le K(a,b,c)q_{0}^{2}q_{i_{0}}  \\
	&\le K(a,b,c) q_{0}^{2}C_{1}P_{m-2}^{1/6}(q_{0}^{3}K(a,b,c))^{1/5}(ab)^{4/5}.
\end{align*}
We then use Lemma \ref{lem:tool2} (i) to bound $q_0$, obtaining
\begin{align}
\label{eqn:boundc}
	(m-3)c	&<C_{1}P_{m-2}^{1/6}K(a,b,c)^{6/5} (ab)^{4/5}\left(4^{2/3}C_{0}(P_{m-2}P_{c})^{1/6}(ab)^{2/3}\right)^{13/5}\\
&< 4^{26/15}C_{0}^{13/5}C_{1}P_{m-2}^{3/5}P_{c}^{1/2}K(a,b,c)^{6/5}(ab)^{38/15}.\nonumber
	\end{align}
Hence if there exists some $j\leq i_0$ for which $q_j\nmid c$, then we have \eqref{eqn:boundc}, which implies the claim. 

On the other hand, if no such $j$ exists, then we have $ q_{1}q_{2}\cdots q_{i_{0}} \mid P_{c}$. We claim that for every $i\geq i_0$ we have $q_{1}q_{2}\cdots q_i\mid P_c$, leading to a contradiction because $c$ is finite. The case $i=i_0$ is assumed, and we proceed by induction. Suppose that $i\geq i_0$ and $q_1\cdots q_{i}\mid P_c$. If $q_{i+1}\nmid c$, then we again have $n_{i+1}\leq K(a,b,c)q_0q_{i+1}$ by Lemma \ref{lem:Ni} and repeating the above argument we obtain 
	\begin{align*}
	(m-3)q_{1}q_{2}\cdots q_{i}\le (m-3)P_{c} \le (m-3)c\le n_{i+1}\le K(a,b,c)q_{0}^{2}q_{i+1},
	\end{align*}
	which contradicts the inequality \eqref{eq34} in Lemma \ref{lem:tool2}. We conclude that $j$ must exist, and therefore \eqref{eqn:boundc} follows.
\end{proof}

\noindent\textbf{Proof of Theorem \ref{thm:noregularmgonalforms}.}  
First suppose that $\triangle_{m,(a,b,c)}$ is primitive and regular and $G_m(a,b,c)\neq B_{m}(a,b,c)$. By Lemma \ref{lem:wt_mgonalforms}, there exists another primitive regular form $\triangle_{m,(a',b',c')}$ with $G_m(a',b',c')=B_{m}(a^{\prime},b^{\prime},c^{\prime})$. It thus suffices to prove that there do not exist any primitive regular forms with $G_m(a,b,c)=B_{m}(a,b,c)$ for $m$ sufficiently large.

Assume that $ a\le b\le c $, $ \gcd(a,b,c)=1 $ and $ G_{m}(a,b,c)=B_{m}(a,b,c) $. Note that $ P_{ab}\le ab $ and $ P_{c}\le c  $. Also, when $ m\ge 6 $, $ P_{m-2}^{3/5}<(m-3)^{4/5}$. By Lemma \ref{lem:bound_c}, 
\begin{align*}
(m-3)c&\le C_{3}P_{m-2}^{3/5}P_{c}^{1/2}K(a,b,c)^{6/5}(ab)^{38/15} \\
&< C_{3}(m-3)^{4/5}c^{1/2}(24P_{ab}\rho(P_{abc}))^{6/5}(ab)^{38/15} \\
&< C_{3}(24\rho(P_{abc}))^{6/5}(m-3)^{4/5}c^{1/2}(ab)^{56/15}.
\end{align*}
Therefore, $ (m-3)^{2/5}c<C_{3}^{2}(24\rho(P_{abc}))^{12/5}(ab)^{112/15} $.
By Lemma \ref{lem:bound_ab}, we deduce that
\begin{align*}
(m-3)^{2/5}c&<C_{3}^{2}(24\rho(P_{abc}))^{12/5}(ab)^{112/15} \\
&< C_{3}^{2}(24\rho(P_{abc}))^{12/5}(C_{2}P_{abc}\rho(P_{abc})^{2}/\phi(P_{abc}))^{112/15}.
\end{align*} 
and so 
$(m-3)^{2/5}c<C_{4}\rho(P_{abc})^{18}(P_{abc}/\phi(P_{abc}))^{112/15}$, where $ C_{4}=24^{12/5}C_{2}^{112/15}C_{3}^{2}$. Since $ P_{abc}/\phi(P_{abc})\ll P_{abc}^{\varepsilon} $ and $ 2^{\omega(P_{abc})}\ll P_{abc}^{\varepsilon} $, 
\begin{align*}
(m-3)^{2/5}a\le (m-3)^{2/5}b\le (m-3)^{2/5}c\ll P_{abc}^{\varepsilon}.
\end{align*}
This implies that $ (m-3)^{6/5}abc\ll P_{abc}^{\varepsilon} $, where the implied constant only depends on $ \varepsilon $, but not on $ m $. Since $1\leq a\leq b\leq c$, this leads to a contradiction for $m$ sufficiently large.

\section*{Acknowledgments}
The authors would like to thank Yuk-Kam Lau for helpful discussion and the referee for his/her useful comments and suggestions.

\end{document}